\documentclass[12pt]{amsart}
\usepackage[shortlabels]{enumitem}
\usepackage{bm}
\usepackage{color, upgreek}
\usepackage{amssymb}

\usepackage{phaistos}
\usepackage{fancyhdr,color,graphicx}
\usepackage{tikz}
\usepackage{pgfplots}
\usepackage{psfrag}
\usepackage{multicol}
\usetikzlibrary{arrows,shapes,positioning}
\graphicspath{{graphics/}}

\tikzset{
    cross/.pic = {
    \draw[rotate = 30] (-#1,0) -- (#1,0);
    \draw[rotate = 60] (0,-#1) -- (0, #1);
    }
}

\newtheorem{thm}{Theorem}
\newtheorem{rem}[thm]{Remark}
\newtheorem{lem}[thm]{Lemma}
\newtheorem{cor}[thm]{Corollary}
\newtheorem{q}{Question}
\newtheorem{example}[thm]{Example}
\newcommand{\Z}{\mathbb Z}
\newcommand{\N}{\mathbb N}
\newcommand{\R}{\mathbb R}

\newcommand{\bt}{{\bm{t}}}
\newcommand{\bs}{\bm{s}}
\newcommand{\bv}{\textbf{v}}
\newcommand{\bu}{\textbf{u}}
\newcommand{\tv}{\text{v}}
\newcommand{\bn}{\bm{n}}
\newcommand{\bgamma}{{\bm{\gamma}}}
\DeclareMathOperator{\fract}{frac}
\DeclareMathOperator{\Cl}{Cl}
\newcommand{\AQchange}[1]{{#1}}
\newcommand{\AQdelete}[1]{}

\title{Flexibility of the Pressure Function}

\author{Tamara Kucherenko}\address{Department of Mathematics,
The City College of New York, New York, NY, 10031, USA}\email{tkucherenko@ccny.cuny.edu}

\author{Anthony Quas}\address{ Department of Mathematics and Statistics, University of Victoria, Victoria, BC
Canada}\email{aquas@uvic.ca}

\thanks{T.K. is supported by grants from the Simons Foundation \#430032 and from the PSC-CUNY TRADA-48-19. }
\thanks{A.Q. is supported by a grant from NSERC}

\begin{document}

\begin{abstract}
We study the flexibility of the pressure function of a continuous potential (observable) with respect
to a parameter regarded as the inverse temperature. The points of non-differentiability of this function
are of particular interest in statistical physics, since they correspond to phase transitions. It is well known
that the pressure function is convex, Lipschitz, and has an asymptote at infinity. We prove that in a
setting of one-dimensional compact symbolic systems these are the only restrictions. We present a
method to explicitly construct a continuous potential whose pressure function coincides with
\emph{any} prescribed convex Lipschitz asymptotically linear function starting at a given positive value
of the parameter. In fact, we establish a multidimensional version of this result. As a consequence,
we obtain that for a continuous observable the phase transitions can occur at a countable dense set of
temperature values. We go further and show that one can vary the cardinality of the set of ergodic
equilibrium states as a function of the parameter to be any number, finite or infinite.

\end{abstract}

\subjclass[2000]{}
\maketitle

\section{Overview}

Katok launched the flexibility program which he described as follows:
``under properly understood general restrictions, within a fixed class of smooth dynamical systems dynamical invariants take arbitrary values”.
Hence, the flexibility program is geared towards an understanding of the most general constraints
which define a common class of dynamical systems and the building of tools to readily change all
other dynamical specifications within those constraints. This is a novel direction in dynamics which
has been explicitly stated in \cite{EK,BKHR}. At the same time however, the core problems are clear
and accessible to a rather broad community of mathematicians working within the area and this
has made the program develop at a rapid pace. Although Katok originally formulated the program
for smooth dynamical systems, his perception is highly relevant for general topological dynamical
systems on compact spaces. In this note we apply it to the topological pressure functional in the
class of compact symbolic systems.

Within the last few years there has been a great deal of activity around
Katok's ideas of flexibility. We briefly describe some of these works. For conservative Anosov flows
on three-dimensional manifolds the basic flexibility problem involves realization of arbitrary pairs of
numbers as values for topological and metric entropy subject only to the variational inequality.
In \cite{EK} the authors consider smooth closed Riemannian surfaces of negative curvature and
show that all the possible values for the topological and metric (with respect to the Liouville measure)
entropies of the geodesic flow are realized within this class. In situations where the relevant invariant
measure varies with the dynamics, one may be interested in the values of Lyapunov exponents.
The flexibility of Lyapunov exponents was proven for expanding maps on a circle \cite{E1} and for
Anosov area-preserving diffeomorphisms on tori \cite{E2}. Subsequently, the fundamental paper
\cite{BKHR} outlines the program and provides flexibility results for volume-preserving systems
with respect to the volume measure. These results have already been improved and extended in
\cite{CS}.

There are also applications of the flexibility paradigm in settings other than smooth flows on manifolds.
The class of piecewise expanding unimodal maps is considered in \cite{AMP}. The authors show that
the only restrictions for the values of the topological and metric entropies in this class are that both are
positive and the topological entropy is at most log 2. In \cite{AKU} some maps arising in the study of
Fuchsian groups are analyzed  and it is proven that all possible values of the entropy are attained.
Lastly, flexibility results are established for the values of polynomial slow entropy for rigid
transformations \cite{BKW} and homeomorphisms on a continuum \cite{RRS}.

The path to obtaining the full range of allowable parameters opened by Katok and then followed by
others consists of starting from a map whose dynamics is well understood and studying what happens
under perturbations. The main challenge of this approach is that the values of dynamical invariants
can be precisely calculated in only a handful of cases. Moreover, there are not many methods available
to perturb a system in a controlled manner.

Establishing flexibility calls for versatile constructions in large families to cover all possible values
of dynamical quantities. This is precisely the route we take to gain total control over one of the
most important objects in thermodynamic formalism. Our work asserts flexibility of a whole pressure
function, rather than of finite number of values for Lyapunov exponents or topological and metric entropies.
In contrast to the perturbation methods described above, we build a dynamical system with the desired
properties from the ground up.  We remark that the pressure function can, in turn, be applied to obtain
information about Lyapunov exponents, dimension, multifractal spectra, or natural invariant measures.
We refer to  \cite{BG,P,PU, Ru1} for details and further references.

Our setting is one-dimensional compact symbolic systems. One of the many reasons that symbolic
systems are important is that they serve as proxies for smooth systems. In many occasions it is more
straightforward to identify properties of symbolic systems, which then can be transferred to smooth systems.
We investigate the possible behavior of the topological pressure restricted to a linear span of a fixed finite
set of continuous potentials. The pressure is then viewed as a function of the coefficients in the linear
combinations of the potentials. Such multivariable pressure functions play a fundamental role in
multifractal analysis, which studies level sets of asymptotically defined quantities such as Birkhoff
averages and local entropies. A nice overview of the theory can be found in \cite{BPS}. The pressure
function is used as the main tool to compute the dimension spectra of the simultaneous level sets,
see e.g. \cite{BSS} and \cite{C}.

To be precise, let $\phi:X\to\R$ be a continuous potential associated with a compact symbolic dynamical system
$(X,\sigma)$. The \emph{topological pressure} of $\phi$ can be defined via the
Variational Principle by
\begin{equation}\label{eq:VarPr}
  P_{\rm top}(\phi)=\sup\left\{h(\mu)+\int\phi\,d\mu \right\}
\end{equation}
where the supremum is taken over the set of all $\sigma$-invariant probability measures on $X$ and $h(\mu)$
denotes the measure-theoretic entropy of the measure $\mu$. The measures which realize the above supremum
are called the \emph{equilibrium states} of $\phi$. Classical manuscripts about the pressure and equilibrium states
are \cite{Bo, Ru1, Walters}.

Fix $m$ continuous potentials $\phi_1,...,\phi_m$. For $(t_1,...t_m)\in\R^m$ the \emph{multivariable pressure function} is the map
$$(t_1,...,t_m)\mapsto P_{\rm top}(t_1\phi_1+...+t_m\phi_m).$$
We now describe a few basic properties of this map. It is an immediate consequence of the Variational Principle
that the pressure function is Lipschitz and convex. The defining characteristic of a convex function on $\R^m$ is
that it has a supporting hyperplane at each point of its graph. \AQchange{It follows from the description of the
equilibrium states as tangent functionals to the pressure given by Walters \cite{W1} that each such hyperplane
arises as the graph of a function $(t_1,\ldots,t_m)\mapsto h(\mu)+\int(t_1\phi_1+\ldots
+t_m\phi_m)\,d\mu$ for an equilibrium state $\mu$. The \emph{vertical intercept} (i.e.~the value of the function
evaluated at $(0,\ldots,0)$) of such a hyperplane is $h(\mu)$.} Note that if the potential $t_1\phi_1+...+t_m\phi_m$ has multiple equilibrium measures, they may correspond to different hyperplanes, all passing through the point $(t_1,\ldots,t_m,P_{\rm top}(t_1\phi_1+...+t_m\phi_m))$. Lemma \ref{lem:ConvAnal} below states that a convex function of the type that we are considering may be recovered as the supremum of its tangent functionals.
\AQdelete{Specifically, a supporting hyperplane
to the pressure function at point $(t_1,...,t_m)$ intersects the vertical axis at $h(\mu)$, where $\mu$ is an
equilibrium state for the potential $t_1\phi_1+...+t_m\phi_m$.} The entropies of all invariant probability measures
are bounded above by the topological entropy of the system $(X,\sigma)$. Hence, if a real valued function of $m$
variables is a pressure function then it is convex, Lip\-schitz, and the vertical intercepts of its supporting
hyperplanes form a bounded set of nonnegative numbers. We prove that these conditions are
\emph{necessary and sufficient}. Our main result is the following theorem.

\begin{thm}\label{thm:main}
Let $\alpha>0$ and let $F(t_1,\ldots,t_m)$ be a convex Lipschitz function on $(\alpha,\infty)^m$ such
that all the supporting hyperplanes to the graph of $F$ intersect the vertical axis in a closed interval $[b,c]\subset [0,\infty)$.
Then there exists a full shift on a finite alphabet and continuous potentials $\phi_1,\ldots,\phi_m$ such that
$P_{\rm top}(t_1\phi_1+\ldots+t_m\phi_m)=F(t_1,\ldots,t_m)$ for all $(t_1,\ldots,t_m)\in (\alpha,\infty)^m$.
\end{thm}

Our proof is explicit and constructive. For an arbitrary function $F$ satisfying these
properties we build a set of $m$ continuous potentials whose pressure function coincides with $F$ on $(\alpha,\infty)^m$.  Theorem \ref{thm:main} falls in line with Katok's flexibility program.
Within the class of full shifts on finite alphabets we identify the general constraints on the pressure function and provide a tool to acquire any pressure function within those constraints.

We comment on the reasons behind our choice of the domain of the function $F$. Among symbolic systems our focus is the full shifts. Since the pressure of the constant zero potential equals the entropy of the system, the presence of the origin in the domain of $F$ would not only impose an additional restriction on $F$, i.e. $F(0)$ must be a logarithm of an integer $d\ge 2$, but also fix the dynamical system itself, i.e. $X$ must be the full shift on $d$ symbols. This leads to a very interesting related question of which functions on $\R^m$ can occur as pressure functions for a specific dynamical system, in particular the full shift on $d$ symbols. The task of identifying such functions might turn out to be very complex especially if the set of such functions does not admit a ``nice" description. We are not aware of any results in this direction. In order to operate freely within the class of full shifts the domain of $F$ should not contain the origin in its convex hull. Considering the first hyperoctant of $\R^m$ as the domain of $F$ seems the most natural to us, especially since in dimension one the positive parameter $t$ has a physical interpretation as the inverse temperature of the system. 

In the case when $X$ is a transitive subshift of finite type and the potentials $\phi_1,...,\phi_m$ are H\"older
the pressure function $P_{\rm top}(t_1\phi_1+...+t_m\phi_m)$ is analytic. This fact goes back to the results
of Bowen \cite{Bo} and Ruelle \cite{Ru1,Ru2}. Starting with an analytic function $F(t_1,...,t_m)$ we obtain
from Theorem \ref{thm:main} a set of continuous potentials for which the pressure function coincides with $F$.
However, our potentials are not H\"older. This raises another interesting question of whether an analog of Theorem
\ref{thm:main} holds in the case when the potentials are required to be H\"older,
i.e. whether any analytic convex function is a pressure function for a set of H\"older continuous potentials.

We briefly outline the ideas which go into the proof of Theorem \ref{thm:main}. For simplicity, we consider
a one-parameter pressure function here, i.e., $m=1$. We start with a convex Lipschitz function
$F:(\alpha,\infty)\to\R$. By convexity, for each point on the graph of $F$ there is at least one supporting line and,
by our assumption, it intercepts the vertical axis in the interval $[b,c]\subset[0,\infty)$. Figure \ref{pic:main_thm}
below illustrates the setup. As was mentioned before, a supporting line to the pressure function $P(t\phi)$ at
$t$ must have vertical intercept $h(\mu_t)$ and slope
$\int\phi\, d\mu_t$, where $\mu_t$ is one of the equilibrium states of the potential $t\phi$. Our goal is to construct
$\phi$ such that $F(t)=P(t\phi)$. The general idea is that the equilibrium states of $t\phi$ move among a family
of disjointly supported subshifts when $t$ changes. Hence, we need to find a family of subshifts whose entropies
fill up the whole interval $[b,c]$. Good candidates for this purpose are the $\beta$-shifts.
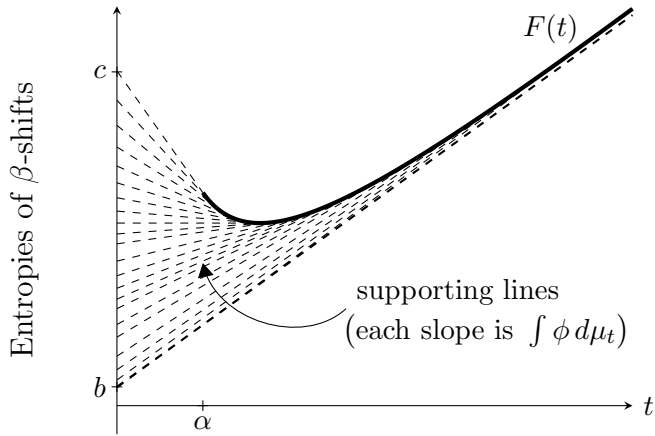
\begin{figure}[h]
\[\begin{tikzpicture}
\begin{axis}[ylabel=Entropies of $\beta$-shifts,
samples=100,
axis y line=left,
axis x line=middle,
ytick=\empty,
xtick=\empty,
clip=false,
xlabel near ticks,
ymax=7, ymin=-0.5,
]
\addplot[smooth,color=black,ultra thick,domain=0.6:3.6]{2*x+1/(0.5*x)-0.8};
\foreach \n in {0.6,0.65,0.7,0.75,0.8,0.85,0.9,0.95,1,1.05,1.1,1.2,1.3,1.4,1.5,1.6,1.8,2, 2.4, 2.8, 3.2, 3.6}
\addplot[smooth,color=black,dashed,domain=0:\n]{(2-2/(\n)^2)*(x-\n)+(2*\n+2/\n-0.8)};
\addplot[smooth,color=black,thick, dashed,domain=0:3.6]{1.81*x+0.33};
\addplot [only marks,mark = +] coordinates {(0.6,0) (0,0.33) (0,5.85)};
\node[below] at (axis cs:0.6,0){\small{$\alpha$}};
\node[left] at (axis cs:0,0.33){\small{$b$}};
\node[left] at (axis cs:0,5.85){\small{$c$}};
\node[left] at (axis cs:3.3,6.6){\small{$F(t)$}};
\node[right] at (axis cs:3.6,0){\small{$t$}};
\node[anchor=west] (source) at (axis cs:1.6,2){\small{supporting lines}};
\node[anchor=west] (source) at (axis cs:1.5,1.3){\small{$\left(\text{each slope is }\int\phi\, d\mu_t\right)$}};
\draw[->,-triangle 60] (axis cs:1.6,1.7) to[bend right=-60] (axis cs:0.6,2.5);
\end{axis}
\end{tikzpicture}\]
\caption{ This figure illustrates the proof of Theorem \ref{thm:main}.}
\label{pic:main_thm}
\end{figure}

The origin of $\beta$-shifts lies in the study of expansions of real numbers in an arbitrary real base $\beta>1$,
which were introduced by Renyi \cite{Re}. Roughly speaking, the $\beta$-shift $X_\beta$ consists of the
sequences of the coefficients in the $\beta$ expansions of reals in $[0,1)$. The measure-theoretic properties
of $\beta$-shifts and their connection to these expansions were initially studied in \cite{Parry, IT, H}. It was
shown that $\{X_\beta;\beta>1\}$ is an increasing family of shift-invariant closed sets with
$h_{\rm top}(X_\beta)=\log\beta$ and $X_\beta$ has a unique measure of maximal entropy.

The entropies of $\beta$-shifts have the properties we need for our construction. The next step
would be to define the potential $\phi$ on each $X_\beta$ as the constant equal to the slope of
the corresponding supporting line. There is an obstacle, however: our subshifts $X_\beta$ are
nested. We avoid it by introducing an additional ``dimension" in the following way. We take a
product of each $\beta$-shift with a suitably chosen Sturmian shift. Sturmian shifts are low
complexity systems with a variety of combinatorial properties useful for our analysis. They
have been studied since the birth of symbolic dynamics \cite{MH}, but modern interest was
sparked by numerous applications in computer science \cite{BBC,CSc,HOSSS,ScZ}.
We refer to Section 4 for definitions and examples of Sturmian shifts.

The low complexity of Sturmian shifts ensures that they do not contribute to the entropy
of the product. We let $\beta$ run from $e^b$ to $e^c$ and obtain a set of disjoint subshifts
of an appropriate full shift which are products of $\beta$-shifts with Sturmian shifts and
whose topological entropies fill the interval $[b,c]$ (recall Figure 1). We are in a position to
define the potential $\phi$ on each such subshift to be the slope of the supporting line to $F(t)$
which crosses the vertical axis at $\log \beta$.

After we make this careful arrangement on all the subshifts, we still need to take care of all the other points in
our full shift. The idea is to make $\phi$ drop off sharply and force the equilibrium measures to be supported
on the products of $\beta$-shifts and Sturmian shifts. This is the most challenging part of the proof.
We accomplish this by using a pin-sequence technique introduced in \cite{Antonioli} to measurably
split orbits of symbolic dynamical systems into finite segments.

Our results have implications for occurrences of phase transitions. A phase transition is
observed when one follows an evolution of a system depending on a continuous external
parameter and a sharp change of the behaviour of the system happens. Understanding the mechanism of
this phenomenon is a fundamental goal in statistical physics. To achieve this, simplified mathematical
models were proposed, the most well known one being the Ising model \cite{Kac, Dob, FS, Sinai},
leading to the development of thermodynamic formalism. In this setting, the quantity
$h(\mu)+\int\phi\,d\mu$ represents the negative free energy of the system in the state $\mu$ with
respect to the observable $\phi$. Hence, the pressure of $\phi$ is the minimum of free energies
and the equilibrium states of $\phi$ characterize the equilibria of the system. The existence of
more than one equilibrium state corresponds to a phase transition.

One way to change the equilibrium state of the system is by adding heat. The measure of temperature in thermodynamics is the
absolute temperature $T$, which is always a positive number with the limit $T\to0^+$ being absolute zero.
Hence, a positive parameter $t=1/T$ (the inverse temperature of the system) is introduced and one studies
how the equilibrium states of $t\phi$ change with $t$, identifying the values of $t$ for which the potential
$t\phi$ has more than one equilibrium state. A classical result by Walters \cite{W1} is that non-differentiability
of the pressure function $t\to P(t\phi)$ at $t_0$ is equivalent to the potential $t_0\phi$ having two equilibrium
states with distinct entropies. Such points of non-differentiability are called \emph{first-order} phase transitions.
Points where the pressure function is differentiable, but not analytic, are termed \emph{higher-order} phase
transitions. Although non-uniqueness of equilibrium states may not appear at such points, they still
indicate a sharp change in some property of the system.

For symbolic systems the first systematic study of a family of potentials exhibiting a phase transition
at some value $t_0$ was done by Lopes in \cite{Lo1, Lo2, Lo3} building upon the previous work of Hofbauer \cite{H2}.
Consequently, phase transitions in thermodynamic formalism were examined using various approaches, see
e.g. \cite{BLL,BL,CR,DGR,Dobbs, IoTo, Lep}. We note that the main results of \cite{Dobbs} and \cite{Lep} deal with the shape of the pressure
function after the transitional value: an example is built of a unimodal map and a subshift of finite type respectively where a phase transition occurs
but the pressure is strictly convex. Prior to \cite{Lep}, in all known examples with phase transitions on shifts, the pressure function was either flat after
the transition, or there was at least some interval where the pressure was flat \cite{IoTo}.

Another aspect concerns the number and frequency of phase transitions. The existence of
infinitely-renormalisable quadratic maps which admit an infinity of phase transitions was shown in \cite{Dobbs}. In \cite{KQW} we construct a continuous potential on $\{0,1\}^\Z$ whose first-order phase transitions occur at any given increasing sequence.
Up to that point there were no examples in the literature of more than two phase transitions in the compact
symbolic setting. Note that the convexity of the pressure implies that at most countably many points of
non-differentiability are possible. Although we see from \cite{KQW} that the case of infinitely many such
points can indeed be realized for the pressure, the requirement for them to form an increasing sequence is
actually quite restrictive. A convex function, in general, may have a dense set of points where its derivative
does not exist. The question remained whether similar behavior is feasible for the pressure function.
As a consequence of our main result we see that the answer is yes. We provide a method of obtaining
continuous potentials whose phase transitions form any given countable set. In addition, the pressure
function between the phase transitions can be made strictly convex, whereas the pressure in \cite{KQW}
is piecewise linear.

Finally we turn our attention to the type of phase transitions where the pressure function is analytic, but uniqueness of equilibrium states fails.
The first example of a transitive system for which two equilibria co-exist despite the analyticity of the pressure
was given in 2015 by Leplaideur \cite{Lep}. His work made it clear that the hope for high regularity of the pressure
function to ensure uniqueness of the equilibrium state was unfounded.
We show that regularity of the pressure does not impose any limitations on the behavior of the equilibria of the system.
At any smooth point of the pressure function the corresponding potential may have any number of ergodic
equilibrium states, finite or infinite. Moreover, the cardinality of equilibrium states may change drastically
when the values of the parameter change. The next theorem provides a flexible way of constructing systems
of potentials with varying cardinalities of the equilibrium measures. To emphasize the one-parameter setting we use lower case $f$ to denote the desired convex function.

\begin{thm}\label{thm:main+card}
Let $f(t)$ be a strictly convex differentiable function on $(\alpha,\infty)$ with support line intercepts
lying in a bounded interval $[b,c]\subset[0,\infty)$. Then for any $\ell\in\N$ and any upper semi-continuous function
$N\colon (\alpha,\infty)\to\{1,\ldots,\ell,\infty\}$, there
exists a full shift $(X,\sigma)$ and a potential function $\phi$ such that
\begin{itemize}
\item $P(t\phi)=f(t)$ for all $t\in (\alpha,\infty)$;
\item the cardinality of the set of ergodic equilibrium states for $t\phi$ is exactly $N(t)$.
\end{itemize}
\end{thm}

This result contrasts sharply with the case of H\"{o}lder potentials, where the pressure function is analytic
and the equilibrium state is always unique. It also immediately provides examples of the types found in
\cite{IT,KQW,Lep}.

The paper is organized as follows. In Sections \ref{sec:ConvexAnalysis}, \ref{sec:beta-shifts},
and \ref{sec:SturmianShifts} we introduce the terminology and prove preliminary lemmas concerning
convex functions, beta-shifts and Sturmian shifts respectively. Section~\ref{sec:main} is devoted to the
proof of Theorem \ref{thm:main}. In Section \ref{sec:OneParameterPressure} we examine the one-parameter
pressure function and establish a slight strengthening of Theorem \ref{thm:main} in this case. Also, here
we supply a procedure for building potentials with a given countable set of first-order phase transitions.
Section \ref{sec:Cardinality} contains a discussion on the cardinality of equilibrium states and the
proof of Theorem \ref{thm:main+card}. Lastly, in Section \ref{FutureDirections} we state a few related open questions within the flexibility program.

\section{Convex analysis}\label{sec:ConvexAnalysis}
Suppose $F:(\alpha,\infty)^m\to \R$ is a convex function of $m$ variables. A vector $\bv\in\R^m$ is a
\emph{subgradient} of $F$ at $\bm{s}\in (\alpha,\infty)^m$ if for all $\bm{t}\in (\alpha,\infty)^m$ we have
$$F(\bm{t})\ge F(\bm{s})+\bv\cdot(\bm{t}-\bm{s}).$$ Hence, $\bv$ is a subgradient of $F$ at $\bs$ if the
affine function $G(\bt)=F(\bs)+\bv\cdot(\bm{t}-\bm{s})$ is a global underestimator of $F$.
The graph of $G$ is a hyperplane in $\R^{m+1}$ which is called a \emph{supporting hyperplane of $F$ at} $\bs$.
\AQchange{We refer to $G(0)$ as the \emph{vertical axis intercept} of the hyperplane,
so that the intercept is at $F(\bs)-\bv\cdot\bs$.}
\AQdelete{Note that this hyperplane intersects the vertical axis at $F(\bs)-\bt\cdot\bs$ and }
Under the assumptions of Theorem \ref{thm:main} this intercept must lie in the interval $[b,c]$. The set of all
subgradients of $F$ at $\bs$ is called the \emph{subdifferential} of $F$ at $\bs$ and is denoted by $\partial F(\bs)$.
Since $F$ is convex, for any $\bs\in(\alpha,\infty)^m$ the set $\partial F(\bs)$ is nonempty, closed and convex.
Moreover, $\partial F(\bs)$ is a singleton if and only if $F$ is differentiable at $\bs$.

Let $F$ be as in Theorem \ref{thm:main}. Denote by $L$ the Lipschitz constant of $F$. We define
\begin{equation}\label{def:S}
  S=\Cl\left(\bigcup_{\bs\in(\alpha,\infty)^m}\{(F(\bs)-\bv\cdot\bs,\bv):\bv\in\partial F(\bs)\}\right).
\end{equation}
Then $S$ is a bounded subset of $\R^{m+1}$. In fact, $S\subset [b,c]\times[-L,L]^m$.

\begin{lem}\label{lem:ConvAnal}
Let $F$ be as in Theorem \ref{thm:main} and let the set $S$ be as defined above.
For each $\bt\in(\alpha,\infty)^m$,
$$
F(\bt)=\sup_{(h,\normalfont{\bv})\in S}(h+\normalfont{\bv}\cdot\bt).
$$
\end{lem}

\begin{proof}
\AQdelete{Let $\bt\in (\alpha,\infty)^m$. Take $\bu\in\partial F(\bt)$. Note that this set is
non-empty since $F$ is convex. Then for any $\bs\in(\alpha,\infty)^m$
$$
F(\bs)\ge F(\bt)+\bu\cdot\,(\bs-\bt)=F(\bt)-\bu\cdot\bt+\bu\cdot \bs.
$$
Hence the point $(p,\bu)$ with $p=F(\bt)-\bu\cdot\bt$ is in the set $S$ and $F(\bt)=p+\bu\cdot \bt$. Therefore,
$$
F(\bt)\le \sup_{(h,\bv)\in S}(h+\bv\cdot\bt).
$$
}

\AQdelete{
On the other hand, by continuity, for all $(h,\bv)\in S$ and all $\bt\in (\alpha,\infty)^m$
we have $F(\bt)\ge h+\bv\cdot\bt$. Taking the supremum we obtain the opposite inequality
$$
F(\bt)\ge \sup_{(h,\bv)\in S}(h+\bv\cdot\bt).
$$
}

\AQchange{
If $\bs\in(\alpha,\infty)^m$ and $\bv\in\partial F(\bs)$ then for any $\bt\in (\alpha,\infty)^m$,
$F(\bt)\ge h+\bv\cdot \bt$ where $h=F(\bs)-\bv\cdot \bs$. Since  $F(\bs)-\bv\cdot \bs$  depends continuously on $\bs$ and $\bv$, the same inequality
holds for any $(h,\bv)\in S$, so that $$F(\bt)\ge \sup_{(h,\bv)\in S}(h+\bv\cdot \bt).$$
}

\AQchange{
Conversely, given $\bt\in (\alpha,\infty)^m$, let $\bv\in\partial F(\bt)$ and $h=F(\bt)-\bv\cdot\bt$
so that $(h,\bv)\in S$. Now $$F(\bt)=h+\bv\cdot \bt\le\sup_{(h,\bf v)\in S}(h+\bv\cdot \bt),$$ establishing the
reverse inequality.
}
\end{proof}

\section{Beta-shifts}\label{sec:beta-shifts}
The $\beta$-shifts, which emerged from the notion of base $\beta$ representation of real numbers \cite{Re},
were first systematically studied as dynamical systems by Parry in \cite{Parry}. For a fixed $\beta>1$
every real $r\in [0,1]$ has a $\beta$-expansion
$$r=\sum_{n=1}^{\infty}r_n\beta^{-n},$$
where $r_n$ are from the set $\{0,1,...,\lfloor \beta\rfloor \}$. Here, and throughout the text, $\lfloor .\rfloor$
denotes the floor function, i.e. $\lfloor\beta\rfloor$ is the largest integer not exceeding $\beta$.
The coefficients $r_n$ of the $\beta$-expansion of $r$ are defined using the $\beta$-\emph{transformation}
$T_\beta(r)=\beta r\pmod 1$; $r_n=\lfloor \beta T^{n-1}_\beta(r)\rfloor$.

Consider the set of all sequences of the coefficients in $\beta$-expansions of real numbers in $[0,1)$.
In the case where $\beta$ is an integer, our convention is to include the point $(\beta000\ldots)$.
The $\beta$-\emph{shift} $X_\beta$ is defined to be the closure of the extension of this set to two sided
sequences. Hence, $X_\beta$ is a subshift of $\{0,...,\lfloor\beta\rfloor\}^\Z$ with shift map $\sigma$.
Renyi \cite{Re} \AQchange{gave a description of $X_\beta$ in terms of the $\beta$-expansion of 1}.
Precisely, \AQdelete{if 1 has infinite $\beta$-expansion (all $\lfloor \beta T^{n-1}_\beta(1)\rfloor\ne 0$)}
we define the \emph{maximal} word $w^\beta$ by $w_n^\beta=\lfloor \beta T^{n-1}_\beta(1)\rfloor$.
\AQdelete{If the $\beta$-expansion of 1 is finite the maximal word $w^\beta$ is periodic with period
being the smallest $m$ such that $\lfloor \beta T^{m}_\beta(1)\rfloor=0$. In this case
$w_n^\beta=\lfloor \beta T^{n-1}_\beta(1)\rfloor$ for $n<m$ and $w^\beta_m=\lfloor \beta T^{m-1}_\beta(1)\rfloor-1$.}
The sequence $(r_n)_{n=1}^\infty$ corresponds to a $\beta$-expansion of some $r\in[0,1)$ if and only if
for all $j\in\N$ the word $\sigma^j(r_1r_2...)$ is smaller than $w^\beta$ according to the lexicographical order.

It is well known that the topological entropy of $X_\beta$ is $\log\beta$ (the proof can be found in \cite{Re,IT,Walters78}).
\AQdelete{This makes $\beta$-shifts one of the few natural classes of dynamical systems where the entropy can
assume any desired positive real value.} In addition, results of Hofbauer \cite{Hofbauer} and Walters \cite{Walters78}
show that $\beta$-shifts are intrinsically ergodic.
The unique measure of maximal entropy of $X_\beta$ is weak-mixing \cite{Parry} and Bernoulli \cite{Smorodinsky}.

We need some facts about the language of a $\beta$-shift.
As usual, let $\mathcal L_n(X_\beta)$ denote the set of words of length $n$ forming
sub-words of elements of $X_\beta$ and let $\mathcal L(X_\beta)=\bigcup_n
\mathcal L_n(X_\beta)$ be the \emph{language} of $X_\beta$. Some of the calculations in
the next lemma may also be found in Walters' book \cite[page 178]{Walters}.

\begin{lem}\label{lem:betacount}
Let $\beta>1$ and let $X_\beta$ denote the $\beta$-shift.
Then $$\beta^n\le |\mathcal L_n(X_\beta)|\le \frac{\beta}{\beta-1}\beta^n.$$
\end{lem}

\begin{proof}
Fix $\beta>1$ and let $N_n=|\mathcal L_n(X_\beta)|$. Since $X_\beta$ has entropy
$\log\beta$, sub-multiplicativity of $N_n$ implies that $N_n\ge \beta^n$.

For the opposite inequality let $w^\beta$ be the maximal word for $X_\beta$\AQdelete{, as described above}.
It follows \AQchange{from the description of $X_\beta$ above}
 that an arbitrary element of $\mathcal L(X_\beta)$ is a concatenation of \emph{sub-prefixes} of $w^\beta$, where
a sub-prefix is a word $u$ of some length $k$ such that $u_i=w_i^\beta$ for $i=1,\ldots,k-1$ and $u_k<w_k^\beta$
followed by a (possibly empty) initial segment of $w^\beta$.
In particular, an element of $\mathcal L_n(X_\beta)$ is either the length $n$ prefix of $w^\beta$, or it is a
sub-prefix of some length $j<n$ followed by an arbitrary element of $\mathcal L_{n-j}(X_\beta)$.
Finally, we observe that there are $w_j^\beta$ sub-prefixes of $w^\beta$ of length $j$.
Hence we see
$$
N_n=1+\sum_{j=1}^n w_j^\beta N_{n-j},
$$
where $N_0$ is taken to be 1.
Write $p_n=w^\beta_n/\beta^n$ (so that the $p_n$'s sum to 1);
and $m_n=\max_{j\le n}N_j/\beta^j$.
Dividing the above equation through by $\beta^n$, we obtain
$$
m_n\le \beta^{-n}+\sum_{j=1}^n p_jm_{n-1}\le \beta^{-n}+m_{n-1}.
$$
Therefore, $m_n\le \sum_{j=0}^\infty \beta^{-j}=\frac\beta{\beta-1}$, so that
$N_n\le \frac{\beta}{\beta-1}\beta^n$ as required.
\end{proof}

The following strengthening of the nesting property of $\beta$-shifts is one of the ingredients in the proof
of the main theorem. Although it can be found in the literature (see e.g. \cite{IT}), we give a short proof
here for the sake of completeness.

\begin{lem}\label{lem:BetaRightCts}
Let $\beta>1$. Then
$$
\bigcap_{\beta'>\beta}X_{\beta'}=X_\beta.
$$
\end{lem}

\begin{proof}
Let $\beta$ and $n$ be fixed.
By definition of $w^\beta$, $T_\beta^j(1)<(w^\beta_j+1)/\beta$ for $j=0,\ldots,n-1$. For $\beta'>\beta$,
it is straightforward to see $T_{\beta'}^j(1)>T_\beta^j(1)$ provided $T_{\beta'}^i(1)<(w^\beta_i+1)/\beta'$
for $i=0,\ldots,j-1$. Since the condition $T_{\beta'}^j(1)<(w^\beta_j+1)/\beta'$ is satisfied
on a small interval to the right of $\beta$, we see that for any $n\in\N$,
there exists an interval $[\beta,\beta+\delta_n)$ on which $w^\beta_j=w^{\beta'}_j$ for $j=0,\ldots,n-1$.
The conclusion follows.
\end{proof}

\section{Sturmian shifts}\label{sec:SturmianShifts}
Sturmian shifts were introduced by Morse and Hedlund in \cite{MH} as symbolic coding of geodesic
trajectories on a flat torus. This makes them one of the earliest general classes of shift spaces studied in dynamics.
The most interesting property of these shifts is, undoubtedly, their low complexity. A non-periodic
Sturmian shift not only has zero entropy, it has the smallest growth rate of blocks possible for infinite
shift spaces \cite{CH}. In addition, Sturmian shifts are minimal \cite{Hed} and uniquely ergodic \cite{Bosh}.

While Sturmian words are generally based on the alphabet $\{0,1\}$, we allow Sturmian words with alphabet $\{
\lfloor \gamma\rfloor, \lceil\gamma\rceil\}$ for any $\gamma$. 
We recall that $\lfloor \gamma \rfloor$ is the largest integer not exceeding $\gamma$, while $\lceil \gamma \rceil$
denotes the smallest integer greater than or equal to $\gamma$, and
$\fract(\gamma)=\gamma-\lfloor \gamma\rfloor$ is the fractional part of $\gamma$.
Given $\gamma\in\R$, we first form the sequence $(y^\gamma_i)_{i=-\infty}^\infty$ by
$y^\gamma_i=\lfloor (i+1)\gamma\rfloor-\lfloor i\gamma\rfloor$, with symbols $\lfloor\gamma\rfloor$
and $\lceil \gamma\rceil$. The Sturmian space $Y_\gamma$
is the orbit closure of $y^\gamma$, that is $\Cl(\{\sigma^n(y^\gamma)\colon n\in\Z\})$.
A \emph{Sturmian word} with \emph{slope} $\gamma$ is an element of $\mathcal L(Y_\gamma)$.
A Sturmian word is an element of $\bigcup_{\gamma\in\R} \mathcal L(Y_\gamma)$.

Our terminology comes from the following geometric interpretation of a Sturmian sequence, which we
illustrate in Figure \ref{pic:sturm}. For $\gamma\in\R$ we draw a line with slope $\gamma$ through the
origin on a square grid. Moving from left to right we record the number of times our line intercepts the
horizontal grid lines in each strip between two consecutive vertical grid lines. These numbers form the
corresponding sequence $y^\gamma$. Hence, we can “read off” the Sturmian word $y^\gamma$ from
the graph of the line with slope $\gamma$.
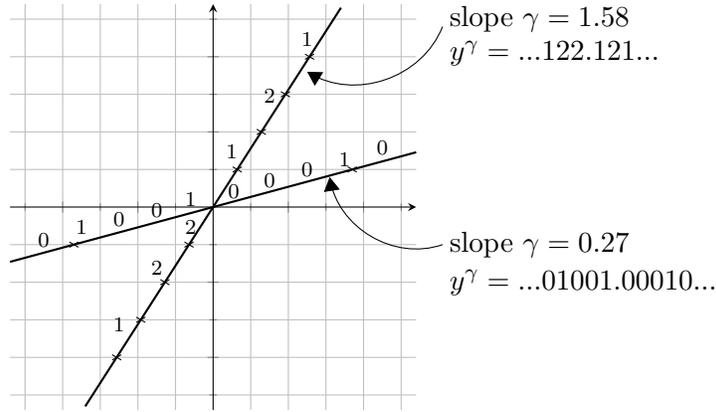
\begin{figure}[h]
\begin{tikzpicture}
  \begin{axis} [grid=both, clip=false,
  x=0.5cm,y=0.5cm,
  axis y line=center,
  axis x line=middle,
  xticklabels=\empty, yticklabels=\empty,
  minor tick num=1,
  ymin=-5.4, ymax=5.4,
  xmin=-5.4,xmax=5.4
  ]
    \addplot [domain=-3.4:3.4, samples=100, thick] { 1.56*x };
    \addplot [domain=-5.4:5.4, samples=100, thick] { 0.27*x };
    \node[anchor=west] at (axis cs:6,5){\small{slope $\gamma=1.58$}};
    \node[anchor=west] at (axis cs:6,4.1){\small{$y^{\gamma}=...122.121...$}};
    \draw[->,-triangle 60] (axis cs:6.1,4.8) to[bend right=-50] (axis cs:2.5,3.6);
    \node[anchor=west] at (axis cs:6,-1){\small{slope $\gamma=0.27$}};
    \node[anchor=west] at (axis cs:6,-2){\small{$y^{\gamma}=...01001.00010...$}};
    \draw[->,-triangle 60] (axis cs:6.1,-1) to[bend right=-50] (axis cs:3.1,0.8);
    \path (axis cs:-2.564,-4) pic[black] {cross=2pt};
    \path (axis cs:-1.923,-3) pic[black] {cross=2pt};
    \path (axis cs:-1.282,-2) pic[black] {cross=2pt};
    \path (axis cs:-0.641,-1) pic[black] {cross=2pt};
    \path (axis cs:0.641,1) pic[black] {cross=2pt};
    \path (axis cs:1.282,2) pic[black] {cross=2pt};
    \path (axis cs:1.923,3) pic[black] {cross=2pt};
    \path (axis cs:2.564,4) pic[black] {cross=2pt};
    \node[anchor=south] at (axis cs:0.5,1){\tiny{1}};
    \node[anchor=south] at (axis cs:1.5,2.5){\tiny{2}};
    \node[anchor=south] at (axis cs:2.5,4){\tiny{1}};
    \node[anchor=south] at (axis cs:-0.6,-1){\tiny{2}};
    \node[anchor=south] at (axis cs:-1.5,-2.1){\tiny{2}};
    \node[anchor=south] at (axis cs:-2.5,-3.6){\tiny{1}};
     \path (axis cs:3.7,1) pic[black] {cross=2pt};
     \path (axis cs:-3.7,-1) pic[black] {cross=2pt};
     \node[anchor=south] at (axis cs:0.55,-0.055){\tiny{0}};
    \node[anchor=south] at (axis cs:1.5,0.25){\tiny{0}};
    \node[anchor=south] at (axis cs:2.5,0.54){\tiny{0}};
    \node[anchor=south] at (axis cs:3.5,0.8){\tiny{1}};
    \node[anchor=south] at (axis cs:4.5,1.1){\tiny{0}};
    \node[anchor=south] at (axis cs:-0.6,-0.25){\tiny{1}};
    \node[anchor=south] at (axis cs:-1.5,-0.55){\tiny{0}};
    \node[anchor=south] at (axis cs:-2.5,-0.8){\tiny{0}};
    \node[anchor=south] at (axis cs:-3.5,-1){\tiny{1}};
    \node[anchor=south] at (axis cs:-4.5,-1.35){\tiny{0}};
  \end{axis}
\end{tikzpicture}
\caption{Geometric interpretation of Sturmian words.}
\label{pic:sturm}
\end{figure}

In the next lemma we characterize the elements of $\mathcal L_n(Y_\gamma)$ and $Y_\gamma$ using a vertical axis intercept of the line with slope $\gamma$.
\begin{lem}\label{lem:SturmChar}
Let $\gamma\in\R$ be fixed. A word $y_0\ldots y_{n-1}$ belongs to $\mathcal L_n(Y_\gamma)$ if and only if
there exists $a\in [0,1)$ such that $y_i=\lfloor (i+1)\gamma+a\rfloor - \lfloor i\gamma+a\rfloor$
for $i=0,\ldots,n-1$.  In particular if $y_0\ldots y_{n-1}$ belongs to $\mathcal L_n(Y_\gamma)$,
then $y_0+\ldots+y_{n-1}$ is either $\lfloor n\gamma\rfloor$ or $\lceil n\gamma\rceil$.

A sequence $(y_i)_{i=-\infty}^\infty$ belongs to $Y_\gamma$ if and only if
\begin{enumerate}[(a)]
\item\label{it:floor} there exists $a\in[0,1)$ such that
$y_i=\lfloor \gamma(i+1)+a\rfloor -\lfloor \gamma i+a\rfloor$ for each $i\in\Z$; or
\item \label{it:ceil}
there exists $a\in (0,1]$ such that $y_i=\lceil \gamma(i+1)+a\rceil - \lceil \gamma i+a\rceil$ for each $i\in\Z$.
\end{enumerate}
\end{lem}
Clearly if $a+i\gamma\not\in\Z$ for all $i$, then the two sequences described in the lemma are equal.
\begin{proof}
If $\gamma$ is rational, $y^\gamma$ is periodic and there is nothing to prove so we suppose
$\gamma$ is irrational.

We first establish the characterization of words.
If there exists $a$ such that $y_i=\lfloor\gamma(i+1)+a\rfloor - \lfloor\gamma i+a\rfloor$ for
$i=0,\ldots,n-1$, then let
$k$ be such that $a<\fract(k\gamma)<a+\min\{1-\fract(a+\gamma i)\colon 0\le i\le n\}$
(such a $k$ exists since the multiples of $\gamma$ are dense modulo 1). One can then check
$y_i=y^\gamma_{k+i}$ for $i=0,\ldots,n-1$. The converse is immediate. Now
if $y$ is of this form, $y_0+\ldots+y_{n-1}=\lfloor a+n\gamma\rfloor - \lfloor a\rfloor$,
which is either $\lfloor n\gamma\rfloor$ or $\lceil n\gamma\rceil$ as required.

We then establish the characterization of $Y_\gamma$.
First suppose $y=\lim_{k\to\infty}\sigma^{n_k}y^\gamma$.
By refining the subsequence if necessary, we may assume that we are in one of the two cases
(i) $\fract (n_k\gamma)$ is a non-increasing sequence converging to some $a\in [0,1)$;
or (ii) $\fract (n_k\gamma)$ is a strictly increasing sequence converging to some $a\in (0,1]$.
For case (i), we note that $\lfloor \cdot\rfloor$ is right continuous, so that for each $i$,
$\lfloor (n_k+i+1)\gamma\rfloor - \lfloor (n_k+i)\gamma\rfloor\to
\lfloor (i+1)\gamma+a\rfloor -\lfloor i\gamma+a\rfloor$, establishing (a).
In the second case, the fact that $\fract(n_k\gamma)$ is strictly increasing implies
that each $n_k$ only appears once, so that $|n_k|\to\infty$.
It follows that for any $i$, $\lfloor (n_k+i+1)\gamma\rfloor-\lfloor (n_k+i)\gamma\rfloor=
\lceil (n_k+i+1)\gamma\rceil-\lceil (n_k+i)\gamma\rceil$ for all sufficiently large $k$
(the only integer multiple of $\gamma$ is $0$). Then since $\lceil\cdot\rceil$
is left continuous, a similar argument to the one above ensures that $y$ satisfies (b).

For the converse, if $y$ satisfies (a), then let $(n_k)$ be chosen so that $\fract(n_k\gamma)$
decreases to $a$. Then $\sigma^{n_k}y^\gamma$ converges to $y$. Similarly, if $y$ satisfies (b), then
choosing $(n_k)$ such that $\fract(n_k\gamma)$ increases to $a$ ensures
$\sigma^{n_k}y^\gamma$ converges to $y$.
\end{proof}

The \emph{weight} of a
Sturmian word $y_0\ldots y_{j-1}$ is $y_0+\ldots+y_{j-1}$. We need the following crude bound
on the number of Sturmian words of a given weight with a fixed length.

\begin{lem}\label{lem:Sturm}
For any $j$ and $n$, there are at most $j(j+1)$ Sturmian words
of length $j$ and weight $n$.
\end{lem}

\begin{proof}
By Lemma \ref{lem:SturmChar}, a Sturmian word is parameterized by an intercept $a$ and a slope $\gamma$.
To satisfy the constraint on the weight, we require $0\le a<1$ and $n\le a+j\gamma<n+1$.
That is, one is looking for a straight line joining a point $(0,a)$ to a point $(j,n+b)$ with $0\le a,b<1$.
Such lines all lie within the parallelogram $y-\frac njx\in [0,1)$, $x\in [0,j]$. We refer the reader to the sketch in Figure \ref{pic:lem6}.
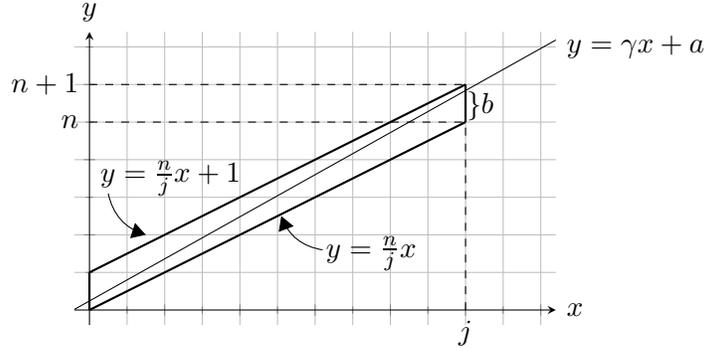
\begin{figure}[h]
\begin{tikzpicture}
  \begin{axis} [grid=both, clip=false,
  x=0.5cm,y=0.5cm,
  axis y line=center,
  axis x line=middle,
  xticklabels=\empty, yticklabels=\empty,
  minor tick num=1,
  ymin=-0.4, ymax=7.4,
  xmin=-0.4,xmax=12.4
  ]
    \addplot [domain=-0.4:12.4, samples=100] { 0.24+0.56*x };
    \draw[thick](axis cs:0,0) -- (axis cs:10,5);
    \draw[thick](axis cs:0,1) -- (axis cs:10,6);
    \draw[thick](axis cs:0,0) -- (axis cs:0,1);
    \draw[thick](axis cs:10,5) -- (axis cs:10,6);
    \draw[dashed](axis cs:10,0) -- (axis cs:10,5);
    \draw[dashed](axis cs:0,5) -- (axis cs:10,5);
    \draw[dashed](axis cs:0,6) -- (axis cs:10,6);
    \node[anchor=north] at (axis cs:10,0){\small{$j$}};
    \node[anchor=east] at (axis cs:0,5){\small{$n$}};
    \node[anchor=east] at (axis cs:0,6){\small{$n+1$}};
    \node[anchor=west] at (axis cs:12.4,0){\small{$x$}};
    \node[anchor=south] at (axis cs:0,7.4){\small{$y$}};
    \node[anchor=west] at (axis cs:6,1.5){\small{$y=\frac{n}{j}x$}};
    \node[anchor=west] at (axis cs:12.4,7){\small{$y=\gamma x+a$}};
    \draw[->,-triangle 60] (axis cs:6.2,1.6) to[bend right=-30] (axis cs:5.1,2.5);
    \node[anchor=west] at (axis cs:0,3.5){\small{$y=\frac{n}{j}x+1$}};
    \draw[->,-triangle 60] (axis cs:0.5,3.1) to[bend left=-30] (axis cs:1.5,2);
    \node[anchor=west] at (axis cs:9.75,5.42){\small{$\}b$}};
  \end{axis}
\end{tikzpicture}
\caption{This figure illustrates the proof of Lemma \ref{lem:Sturm}.}
\label{pic:lem6}
\end{figure}

For each $i$ in $0,\ldots,j$ with $\frac nji\notin\Z$,
there is precisely one integer lattice point in that vertical line within the closure of the
parallelogram, namely $(i,\lceil \frac nj i\rceil)$; or two points $(i,\frac nj i)$ and $(i,\frac nji+1)$
if $\frac nji$ is an integer. Given a Sturmian sequence $y_0\ldots y_{j-1}$ of weight $n$,
let $a$ and $\gamma$ satisfy
\begin{equation}\label{eq:ydef}
y_i=\lfloor (i+1)\gamma+a\rfloor - \lfloor i\gamma +a\rfloor\text{ for $i=0,\ldots,j-1$}.
\end{equation}
To count the number of allowable sequences, we continuously deform $\gamma$ and $a$ without changing the realization of the Sturmian sequence in the coordinates $0,...,j$ until the corresponding line passes through two lattice points. The details are as follows.
One may reduce $a$ keeping \eqref{eq:ydef} satisfied
until the line $y=a+\gamma x$ first hits one of the lattice points, $(i,k)$ say, with $i\in\{0,1,\ldots,j\}$.
One may then rewrite the equation of the line as
$y=k+\gamma(x-i)$, and then reduce $\gamma$ until the line hits
another lattice point in the parallelogram.
The Sturmian word is determined by $i$ together with the $x$-coordinate of the
second lattice point. There are $(j+1)\times j$ such choices, so at most $j(j+1)$ Sturmian words
of length $j$ and weight $n$.
\end{proof}

We now describe the points appearing in the closure of the union of the sets of all Sturmian words \AQchange{over a range of $\gamma$}.

\begin{lem}\label{lem:SturmClose}
Let $Y_\gamma$ denote the Sturmian sequence space with slope $\gamma$ as above.
Suppose $(\gamma^{(n)})_{n\in\N}$ is a sequence of real numbers and
$(y^{(n)})_{n\in \N}$ is a sequence of points with slopes $\gamma^{(n)}$, i.e., $y^{(n)}\in Y_{\gamma^{(n)}}$.
Suppose further that $y^{(n)}\to  y$. Then the sequence $\gamma^{(n)}$ is
convergent to some $\gamma\in \R$.
Either $y\in Y_\gamma$; or
$\gamma$ is rational and $y$ is an aperiodic sequence that is the
concatenation of two periodic semi-infinite words.
\end{lem}

\begin{proof}
Let $\ell\in\N$. Since $y^{(n)}\to y$, there exists an $n_0$ such that for all $n\ge n_0$,
the terms of $y^{(n)}$ in coordinates $-\ell$ to $\ell-1$ agree with those of $y$. Hence
for $n,n'\ge n_0$ the words $y^{(n)}_{-\ell}\ldots y^{(n)}_{\ell-1}$ and
$y^{(n')}_{-\ell}\ldots y^{(n')}_{\ell-1}$  have the same weight, which must lie
simultaneously in the intervals $(2\ell\gamma^{(n)}-1,2\ell\gamma^{(n)}+1)$ and
$(2\ell\gamma^{(n')}-1,2\ell\gamma^{(n')}+1)$ by Lemma \ref{lem:SturmChar}.
It follows that $|\gamma^{(n)}-\gamma^{(n')}|< \frac 1\ell$. Since $\ell$ is arbitrary,
we see that $(\gamma^{(n)})$ is Cauchy. Let $\gamma$ be the limit of $(\gamma^{(n)})$.

Passing to a subsequence, we may assume that $a^{(n)}\to a$. Each $y^{(n)}$ may
be expressed either in the form (a) or (b) of Lemma \ref{lem:SturmChar} with parameters
$a^{(n)}$ and $\gamma^{(n)}$. We may further assume that either each term of the subsequence is expressed
in the form (a); or each term is expressed in the form (b).
Note that for those $i\in\Z$ where $a+i\gamma$ is not
an integer, the sequences $\lfloor a^{(n)}+i\gamma^{(n)}\rfloor$ and
$\lceil a^{(n)}+i\gamma^{(n)}\rceil$ eventually stabilize to $\lfloor a+i\gamma\rfloor$
and $\lceil a+i\gamma\rceil$ respectively.

We deal first with the case where $\gamma$ is irrational. In this case, there is at most one $i_0\in\Z$
such that $a+i_0\gamma\in\Z$. If there is no such $i_0$, then
$\lfloor a^{(n)}+i\gamma^{(n)}\rfloor\to \lfloor a+i\gamma\rfloor$ for each $i$ because $a+i\gamma$ is a continuity point for the floor function for each $i$. A similar statement
is true for the ceilings. It follows that $y_i=\lim y^{(n)}_i=
\lfloor a+(i+1)\gamma\rfloor - \lfloor a+i\gamma\rfloor
=\lceil a+(i+1)\gamma\rceil -\lceil a+i\gamma\rceil$ for each $i$.

Now suppose there exists $i_0$ such that $a+i_0\gamma\in\Z$. Since $a^{(n)}+i_0\gamma^{(n)}$ converges to $a+i_0\gamma$ we can consider three cases:
(i) $a^{(n)}+i_0\gamma^{(n)}=a+i_0\gamma$ infinitely often;
(ii) we can pass to a subsequence such that $a^{(n)}+i_0\gamma^{(n)}$ is strictly  increasing;
(iii) we can pass to a subsequence such that $a^{(n)}+i_0\gamma^{(n)}$ is strictly decreasing.
In case (i), along the subsequence where $a^{(n)}+i_0\gamma^{(n)}=a+i_0\gamma$,
$\lfloor a^{(n)}+i\gamma^{(n)}\rfloor \to \lfloor a+i\gamma\rfloor$ for all $i$ and
$\lceil a^{(n)}+i\gamma^{(n)}\rceil \to \lceil a+i\gamma\rceil$ for all $i$, so that $y\in Y_\gamma$ exactly as in the previous paragraph.
In case (ii), along the subsequence, we have $\lfloor a^{(n)}+i\gamma^{(n)}\rfloor \to \lfloor a+i\gamma\rfloor=\lceil a+i\gamma\rceil -1$ for all $i\ne i_0$ as above. Also, for $i=i_0$ we have $\lfloor a^{(n)}+i_0\gamma^{(n)}\rfloor\to \lfloor a+i_0\gamma\rfloor-1=\lceil a+i_0\gamma\rceil -1$, so that $y\in Y_{\gamma}$ by Lemma \ref{lem:SturmChar}. Similarly in case (iii),
along the subsequence $\lfloor a^{(n)}+i\gamma^{(n)}\rfloor \to \lfloor a+i\gamma\rfloor$ for all $i$.
In all cases, we see $y\in Y_\gamma$.

Now suppose that $\gamma$ is rational, say $\gamma=\frac pq$.
If infinitely many $\gamma^{(n)}$ are equal to $\gamma$ then,
since $Y_\gamma$ is a finite set, we see that one element of $Y_\gamma$ appears infinitely often in the
sequence $y^{(n)}$, so that sequence is the limit and $y\in Y_\gamma$. Otherwise, we may take a sequence
so that $\gamma^{(n)}$ converges strictly monotonically to $\gamma$ and $a^{(n)}$ converges to a limit
$a$. If $a+i\gamma\notin \Z$ for each $i$ (or equivalently $a+i\gamma\notin \Z$
for $i=0,\ldots,q-1$), then the argument given above in the irrational case shows $y\in Y_\gamma$.

In the remaining case, there exists $i_0\in \{0,\ldots,q-1\}$ such that
$a+i\gamma$ is an integer for each $i\in i_0+q\Z$ (and $a+i\gamma$ is not an integer for other $i$'s).
We may then pass to a further
subsequence so that the sequence $j^{(n)}=(a^{(n)}-a)/(\gamma-\gamma^{(n)})$ is monotonic.
If $\gamma^{(n)}$ is increasing, $a^{(n)}+i\gamma^{(n)}<a+i\gamma$ when
$i>j^{(n)}$ and $a^{(n)}+i\gamma^{(n)}>a+i\gamma$
when $i<j^{(n)}$; the situation is reversed if $\gamma^{(n)}$ is decreasing.

For the remainder of the proof, we focus on the case where $\gamma^{(n)}$ is increasing.
If $j^{(n)}\to-\infty$, then for each $i\in\Z$,
$\lceil a^{(n)}+i\gamma^{(n)}\rceil\to \lceil a+i\gamma\rceil$
and $\lfloor a^{(n)}+i\gamma^{(n)}\rfloor \to \lceil a+i\gamma\rceil-1$.
Hence we see that whether the sequence $y^{(n)}$ is expressed in form (a) or form (b),
$y_i= \lceil a+(i+1)\gamma\rceil -\lceil a+i\gamma\rceil$ for all $i\in\Z$.
Similarly if $j^{(n)}\to\infty$, then for each $i\in\Z$, $\lfloor a^{(n)}+i\gamma^{(n)}\rfloor\to
\lfloor a+i\gamma\rfloor$ and $\lceil a^{(n)}+i\gamma^{(n)}\rceil
\to \lfloor a+i\gamma\rfloor+1$, so that $y_i=\lfloor a+(i+1)\gamma\rfloor
-\lfloor a+i\gamma\rfloor$.

We now consider the case $j^{(n)}\to j^*$. We have
$$
\lim_{n\to\infty}\lfloor a^{(n)}+\gamma^{(n)}i\rfloor=
\begin{cases}
\lfloor a+\gamma i\rfloor&\text{if $i<j^*$ or $i\not\in i_0+q\Z$};\\
\lfloor a+\gamma i\rfloor -1&\text{if $i>j^*$ and $i\in i_0+q\Z$}.
\end{cases}
$$
and
$$
\lim_{n\to\infty}\lceil a^{(n)}+\gamma^{(n)}i\rceil=
\begin{cases}
\lceil a+\gamma i\rceil&\text{if $i>j^*$ or $i\not\in i_0+q\Z$};\\
\lceil a+\gamma i\rceil+1&\text{if $i<j^*$ and $i\in i_0+q\Z$}.
\end{cases}
$$
If $j^*\in i_0+q\Z$, then $\lfloor a^{(n)}+\gamma^{(n)}j^*\rfloor$ converges to one of
$\lfloor a+\gamma j^*\rfloor$ and $\lfloor a+\gamma j^*\rfloor-1$; and
$\lceil a^{(n)}+\gamma^{(n)}j^*\rceil$ converges to one of
$\lceil a+\gamma j^*\rceil$ and $\lceil a+\gamma j^*\rceil+1$.

Hence $y$, the difference sequence of one of
$\big(\lim_{n\to\infty} \lfloor a^{(n)}+\gamma^{(n)}i\rfloor\big)_i$
or  $\big(\lim_{n\to\infty} \lceil a^{(n)}+\gamma^{(n)}i\rceil\big)_i$,
is the concatenation of two semi-infinite periodic words, as claimed. In the case where $\gamma^{(n)}$ is
decreasing, an almost identical argument applies.
\end{proof}

\section{Main Theorem}\label{sec:main}

We present the proof of Theorem \ref{thm:main}.
Let the alphabet of the shift be
$$A=\{0,1\ldots,\lfloor e^c\rfloor\}\times \{\lfloor b\rfloor,\ldots,\lceil c\rceil\}\times\{\lfloor -L\rfloor,\ldots,\lceil L\rceil\}^m$$
and let $\sigma$ denote the shift map on $A^\Z$.  The distance between points $z$ and $z'$ in $A^\Z$ is given by
$$d(z,z')=2^{-\inf\{|n|:\  z_n\not=z'_n\}}$$
We construct the potential functions as follows.
For each vector ${\bgamma}=(\gamma_0,\gamma_1,\ldots,\gamma_m)$ in the set $S$ defined in (\ref{def:S}) let
$$
Z_{\bgamma}=X_{e^{\gamma_0}}\times Y_{\gamma_0}\times Y_{\gamma_1}\times \ldots \times Y_{\gamma_m},
$$
where $X_{e^{\gamma_0}}$
is the $\beta$-shift with parameter $\beta=e^{\gamma_0}$ and for $k=0,\ldots,m$
$Y_{\gamma_k}$ is the Sturmian system with angle $\gamma_k$
(so the alphabet of $X_{e^{\gamma_0}}$ is $\{0,1,\ldots,\lfloor e^{\gamma_0}\rfloor\}$ and the alphabet of
$Y_{\gamma_k}$ is $\{\lfloor \gamma_k\rfloor, \lceil\gamma_k\rceil\}$).
In particular, $Z_{\bgamma}\subset A^\Z$. For $z=(x,y^0,\ldots,y^m)\in A^{\Z}$ the projections of $z$ onto
each coordinate are defined by  $\pi_\beta(z)=x$ and $\pi_k(z)=y^k$, $k=0,\ldots,m$.

Let $\mathcal L_{n}(Z_{\bgamma})$ denote the
collection of $n$-words in $Z_{\bgamma}$. For $z\in A^\Z$, we set $j_{\bgamma}(z)=\max\{l\colon z_{-(l-1)}\ldots z_{l-1}\in
\mathcal L_{2l-1}(Z_{\bgamma})\}$,
where $j_\gamma(z)$ is taken to be 0 if $z_0\not\in \mathcal L_0(Z_{\bgamma})$.
For $\bgamma=(\gamma_0,\gamma_1,\ldots,\gamma_m)\in S$ and $k=1,\ldots,m$ we define
$$
\phi_{k,\bgamma}(z)=\gamma_k-\delta_{j_{\bgamma}(z)},
$$
where $\delta_j$ is given by
\begin{equation}\label{eq:delta_j}
  \delta_j=\frac{c + 2L+14+9\log j}{j\min\{\alpha,1\}}
\end{equation}
with $\delta_0=\delta_1+2L$ and $\delta_\infty=0$.
Notice that $(\delta_j)$ form a decreasing sequence, converging to 0.

If $d(z,z')\le 2^{-l}$,
then the two words $z_{-(l-1)}\ldots z_{l-1}$ and $z'_{-(l-1)}\ldots z'_{l-1}$ are equal.
Either both lie in
$\mathcal L_{2l-1}(Z_\bgamma)$, in which case $|\phi_{k,\bgamma}(z)-\phi_{k,\bgamma}(z')|\le \delta_l$
or neither do, in which case $\phi_{k,\bgamma}(z)=\phi_{k,\bgamma}(z')$. Hence we
have shown that for each $k=1,\ldots,m$ the family
$\{\phi_{k,\bgamma}\colon \bgamma\in S\}$ is uniformly equicontinuous.
The potentials $\phi_k$ on $A^\Z$ are then defined by
$$
\phi_k(z)=\sup_{\bgamma\in S}\phi_{k,\bgamma}(z).
$$
The uniform equicontinuity ensures that each $\phi_k$ is continuous.

The proof of the theorem splits into two parts. First, we restrict our considerations to the set
$Z=\bigcup_{\bgamma\in S} Z_{\bgamma}$. On this set we show that the potentials
$\phi_1,\ldots,\phi_n$ have the property we are looking for. Namely, for given values of
parameters $t_k> \alpha$ the pressure of $t_1\phi_1+\ldots+t_m\phi_m$, when restricted
to $\Cl(Z)$, coincides with the value of $F(t_1,\ldots,t_m)$. Afterwards, we demonstrate that
the values of $\phi_k$ outside of the set $Z$ do not contribute to the pressure.

We start by describing the behavior of the potentials $\phi_1,\ldots,\phi_k$ on $Z$.
\begin{lem}\label{lem:phi on Zbeta}
Let $\bgamma=(\gamma_0,\ldots,\gamma_m)\in S$. If $z\in Z_{\bgamma}$, then
$\phi_k(z)=\gamma_k$ for each $k=1,\ldots,m$.
\end{lem}

\begin{proof}
Let $\bgamma=(\gamma_0,\ldots,\gamma_m)\in S$ and let $z=(x,y^0,\ldots,y^m)\in Z_{\bgamma}$.
Fix $1\le k\le m$. From the definition, we see $\phi_{k,\bgamma}(z)=\gamma_k-\delta_\infty=\gamma_k$.
We will show that for any $\bgamma'=(\gamma_0',\ldots,\gamma_m')\in S$ we have $\phi_{k,\bgamma'}(z)\le \gamma_k$.

If $\gamma_k'\le\gamma_k$, then $\phi_{k,\bgamma'}(z)\le \gamma'_k\le \gamma_k$. If $\gamma_k'>\gamma_k$,
then we set $j=\lceil 1/(\gamma_k'-\gamma_k)\rceil$.
Since $z\in Z_{\bgamma}$, we know that $y^k\in Y_{\gamma_k}$ and hence by Lemma \ref{lem:SturmChar} there exists $a\in[0,1)$ such that
$y_i^k=\lfloor (i+1)\gamma_k+a\rfloor -\lfloor i\gamma_k+a\rfloor$ for each
$i=-j,\ldots,j-1$. In particular, $y_{-j}^k+\ldots+y_{j-1}^k=
\lfloor j\gamma_k+a\rfloor -\lfloor -j\gamma_k+a\rfloor$, so that
$y_{-j}^k+\ldots+y_{j-1}^k\in (2j\gamma_k-1,2j\gamma_k+1)$ and this holds for any
$y^k_{-j}\ldots y^k_{j-1}\in \mathcal L_{2j}(Y_{\gamma_k})$.

If the block $y^k_{-j}\ldots y^k_{j-1}$ were also in $\mathcal L_{2j}(Y_{\gamma'_k})$, then similar to the argument above we would have
$y^k_{-j}+\ldots+y^k_{j-1}\in (2j\gamma_k'-1,2j\gamma_k'+1)$.
However, by the choice of $j$, $2j\gamma_k-1\le 2j\gamma'_k+1$,
so that $\mathcal L_{2j}(Y_{\gamma_k})$ and $\mathcal L_{2j}(Y_{\gamma'_k})$ are disjoint.
It follows that $z_{-j}\ldots z_{j-1}\notin \mathcal L_{2j}(Z_{\bgamma'})$ and hence $j_{\bgamma'}(z)\le j$.

Using the facts that $j=\lceil 1/(\gamma_k'-\gamma_k)\rceil$ and $\gamma_k,\gamma'_k\in[-L,L]$ when $\bgamma,\bgamma'\in S$ we obtain
$$
j\le \frac{1}{\gamma_k'-\gamma_k}+1=\frac{1+\gamma'_k-\gamma_k}{\gamma'_k-\gamma_k}\le \frac{1+2L}{\gamma'_k-\gamma_k}.
$$
Therefore,
$$
\phi_{k,\bgamma'}(z)=\gamma'_k-\delta_{j_{\bgamma'}(z)}\le\gamma'_k-\delta_j<\gamma'_k-\frac{1+2L}{j}\le\gamma_k.
$$

\end{proof}

We now turn our attention to the invariant measures on $Z$.

\begin{lem}\label{lem:measures on Z}
  Let $Z=\bigcup_{\bgamma\in S}Z_{\bgamma}$ as above. Then any ergodic invariant measure supported on $\Cl(Z)$
is supported on $Z_{\bgamma}$ for some $\bgamma\in S$.
\end{lem}
\begin{proof}
We first describe points of $\text{Cl}(Z)$. Let $\bar z\in \text{Cl}(Z)$. Then $\bar z$ is the limit of a sequence of points
$z(n)$, where each $z(n)$ belongs to some $Z_{\bgamma(n)}$ with $\bgamma(n)\in S$. Write $z(n)=(x(n),y^0(n),\ldots,
y^m(n))$ and $\bar z=(\bar x,\bar y^0,\ldots,\bar y^m)$. Since $y^k(n)\to \bar y^k$, it
follows from Lemma \ref{lem:SturmClose} that $\gamma_k(n)$ converges to some limit for each $k=0,\ldots,m$.
Let $\bar\bgamma=\lim_{n\to\infty}\bgamma(n)$, so that $\bar\bgamma\in S$.
By Lemma \ref{lem:SturmClose}, for each $0\le k\le m$,
$\bar y^k$ belongs either to $Y_{\bar\gamma_k}$, or is the (non-periodic)
concatenation of two semi-infinite periodic words.

We claim that $\bar x\in X_{e^{\bar\gamma_0}}$.
To see this, first notice that since $z(n)\in Z_{\bgamma(n)}$, we have
$x(n)\in X_{e^{\gamma_0(n)}}$. For any $\beta>\bar\gamma_0$ there is $n_0\in\N$ such that $\gamma_0(n)<\beta$ for $n\ge n_0$.
It follows that $x(n)\in X_{e^\beta}$ for all $n\ge n_0$ and hence $\bar{x}\in X_{e^\beta}$.
Since $\beta>\bar\gamma_0$ is arbitrary, $\bar x$ lies in $\bigcap_{\beta>\bar\gamma_0} X_{e^\beta}=X_{e^{\bar\gamma_0}}$
by Lemma \ref{lem:BetaRightCts}.

Let $C$ denote the (countable) collection of non-periodic concatenations of two semi-infinite
periodic points with symbols in the set
$\{\min(\lfloor b\rfloor, \lfloor -L\rfloor),\ldots,\max(\lceil c\rceil,\lceil L\rceil)\}$.
We have shown that any point of $\Cl(Z)$ either lies in some $Z_{\bgamma}$ with $\bgamma\in S$
or one of its Sturmian coordinates lies in $C$.

Let $\mu$ be an ergodic invariant measure supported on $\Cl(Z)$.
Suppose for a contradiction that
$\mu$ is supported on $\text{Cl}(Z)\setminus Z$.
Then for $\mu$-a.e.\ $z=(x,y^0,\ldots,y^m)$,
there exists a $0\le k\le m$ such that $y^k\in C$. Since $\mu$ is ergodic,
there exists a $k$ such that
for $\mu$-a.e.{} $(x,y^0,\ldots,y^m)$, $y^k\in C$. In particular, the projection
of $\mu$ onto the $k$th
Sturmian factor is supported on a countable set. But this is a contradiction as
countable sets of aperiodic words
do not support any finite invariant measures. Hence $\mu$ is supported on $Z$.

It is left to show that $\mu$ is supported on some $Z_{\bgamma}$. Fix $k\in\{0,\ldots,m\}$ and consider
the projection map $f_k(z)=y^k_0$ where $z=(x,y^0,\ldots,y^m)\in Z$. Since $\mu$ is ergodic and $f$ is continuous,
there is $\gamma_k\in\R$ such that
\begin{equation*}\label{eq:support_mu}
 \frac 1N\sum_{i=0}^{N-1}f(\sigma^i z)=\gamma_k\quad \text{for }\mu\text{-almost all } z\in Z.
\end{equation*}
Suppose that $z\in Z_{\bgamma'}$ for some $\bgamma'\in S$ and satisfies the above. Then an application of
Lemma \ref{lem:SturmChar} gives
$$
\frac 1N\sum_{i=0}^{N-1}f(\sigma^i z)=\frac{y^k_0+\ldots+y^k_{N-1}}{N}=\frac{\lfloor N\gamma_k'+a\rfloor}{N}
$$
for some $a\in [0,1)$. Hence, $\gamma'_k=\gamma_k$ and $\mu$ is supported on $Z_{\gamma}$ with
$\bgamma=(\gamma_0,\ldots,\gamma_m)$.
\end{proof}

\begin{cor}\label{cor:p on Z}
Let $\bt=(t_1,\ldots,t_m)\in (\alpha,\infty)^m$ and $\mu$ be an ergodic shift-invariant measure supported on $\Cl(Z)$.
Then
$$
h(\mu)+\int (t_1\phi_1+\ldots+t_m\phi_m)\,d\mu\le F(t_1,\ldots,t_m).
$$
Further there exists an ergodic measure $\mu_{\bt}$ supported on $Z$ such that
$$
h(\mu_{\bt})+\int (t_1\phi_1+\ldots+t_m\phi_m)\,d\mu_{\bt}= F(t_1,\ldots,t_m).
$$
\end{cor}

\begin{proof}
Let $\mu$ be as in the statement of the corollary. Note that by Lemma \ref{lem:measures on Z},
 $\mu$ is supported on $Z_{\bgamma}$ for some $\bgamma=(\gamma_0,\ldots,\gamma_m)\in S$.
Then $h(\mu)\le h_\text{top}(Z_{\bgamma})=
h_\text{top}(X_{e^{\gamma_0}})+h_\text{top}(Y_{\gamma_0})+\ldots+
h_\text{top}(Y_{\gamma_m})=\gamma_0$
and $\int \phi_k\,d\mu=\gamma_k$ by Lemma \ref{lem:phi on Zbeta}, so that
$$
h(\mu)+\int (t_1\phi_1+\ldots+t_m\phi_m)\,d\mu\le\gamma_0+t_1\gamma_1+\ldots+t_m\gamma_m.
$$
Since $\bgamma=(\gamma_0,\ldots,\gamma_m)\in S$, the last term is bounded by
$F(t_1,\ldots,t_m)$ by Lemma \ref{lem:ConvAnal}, proving the inequality in the statement.

For the equality in the statement, let $\bv=(\text{v}_1,\ldots,\text{v}_m)$ be any
subgradient of $F$ at $\bt=(t_1,\ldots,t_m)$. We set $\gamma_k=\text{v}_k$ for
$k=1,\ldots,m$ and $\gamma_0=F(\bt)-\bv\cdot \bt$. Then
$\bgamma=(\gamma_0,\ldots,\gamma_m)\in S$.
Let $\nu_\bt$ be the measure on $Z_{\bgamma}$
that is the product of the measure of maximal entropy on  the $\beta$-shift $X_{e^{\gamma_0}}$ and
$m+1$ Sturmian measures supported on each of the components
$Y_{\gamma_0},\ldots,Y_{\gamma_m}$ respectively and let $\mu_\bt$ be an ergodic component.
Since $\mu_\bt$ projects in the first factor onto the measure of
maximal entropy on $X_{e^{\gamma_0}}$, we see $h(\mu_\bt)\ge \gamma_0$. On the other hand,
since $\mu_\bt$ is supported on $Z_\bgamma$,
we see $h(\mu_\bt)\le h_\text{top}(Z_\bgamma)=\gamma_0$.

Each potential $\phi_k$ takes
the value $\gamma_k$ on the support of $\mu_{\bt}$, so that
\begin{align*}
  h(\mu_{\bt})+\int t_1\phi_1+\ldots+t_m\phi_m\,d\mu_{\bt} & =\gamma_0+t_1\gamma_1+\ldots+t_m\gamma_m \\
   & =F(\bt)-\bv\cdot\bt+\bv\cdot\bt \\
   & =F(\bt)
\end{align*}
as required.
\end{proof}

\begin{rem}
  We note that the above arguments remain true if the domain of $F$ is chosen to be the entire $\R^m$. It follows that on the shift space $X=\Cl(Z)$ the pressure function of the potentials $\phi_1,...,\phi_m$ coincides with $F$ for all $(t_1,...,t_m)\in \R^m$. Hence we established the following statement:
  \begin{itemize}
    \item[ ] Given a convex Lipschitz function $F:\R^m\to \R$ such that \mbox{vertical} intercepts of its supporting hyperplanes form a closed interval in $[0,\infty)$ there exists a compact symbolic system $X$ and continuous potentials $\phi_1,...,\phi_m$ on $X$ such that\\ $P(t_1\phi_1+...+t_m\phi_m)=F(t_1,...,t_m)$ for all $(t_1,...,t_m)\in \R^m$.
  \end{itemize}
   Within the class of compact symbolic systems the above result provides necessary and sufficient conditions for a function on $\R^m$ to be the pressure function of a set of continuous potentials. The focus of our investigation is the subclass of full shifts on finite alphabets, where an additional restriction on the domain of $F$ or the value of $F$ at the origin becomes necessary. We refer to  Remark \ref{rem:full shift} and Section \ref{FutureDirections} for further discussion.
\end{rem}

We have shown that for each $(t_1,\ldots,t_m)\in(\alpha,\infty)^m$,
$$
F(t_1,\ldots,t_m)=\sup_{\mu}\left\{h(\mu)+\int (t_1\phi+\ldots+t_m\phi_m)\,d\mu\right\},
$$
where $\mu$ runs over all ergodic invariant measures supported on $\Cl (Z)$.
Recall that in the Variational Principle, the pressure is attained if the supremum is taken only over
ergodic invariant measures. In order to complete the proof of the theorem it suffices to show that for all
$(t_1,\ldots,t_m)\in (\alpha,\infty)^m$ and
for each ergodic shift-invariant measure $\mu$ on $A^\Z$ such that $\mu(\Cl(Z)^c)\ne 0$,
one has
\begin{equation}\label{eq:var_pr_ineq}
  h(\mu)+\int (t_1\phi_1+\ldots+t_m\phi_m)\,d\mu< F(t_1,\ldots,t_m).
\end{equation}

To see this, we use a technique introduced by Antonioli in \cite{Antonioli} of \emph{pinning sequences}.
This is part of a more general set of ideas described in the notes \cite{Quas} on
``Coupling and Splicing''. The idea is to take a word in $A^\Z$ and partition it into maximal subwords, each of which belongs to $\mathcal L(Z)$. Our underlying image is that of a clothes line where maximal words in $\mathcal L(Z)$ are between two adjacent pins.

Formally, the \emph{pinning space} is a closed subshift $\Omega$ of $A^\Z\times\{0,1\}^\Z$ defined by the following
conditions. Let $(u,v)\in A^\Z\times\{0,1\}^\Z$. Then $(u,v)\in\Omega$ if and only if
\begin{enumerate}
\item if $i<j$ and $v_{i+1}=\ldots=v_j=0$, then $u_i\ldots u_j\in\mathcal L(Z)$; (Note: there is no
requirement that $v_i=0$).
\item if $i<j$ and $v_i=v_j=1$, then $u_i\ldots u_j\not\in\mathcal L(Z)$.
\end{enumerate}

We denote the shift map on $\Omega$ by $\bar\sigma$ to distinguish it from the shift $\sigma$ on $A^\Z$. We refer to $v$ in the pair $(u,v)\in \Omega$ as a \emph{pinning sequence} for $u\in A^\Z$.

Clearly if $u\in Z$, then $(u,\mathbf 0)\in \Omega$, where $\mathbf 0$ is the sequence of all 0's.
For a fixed $u\in A^\Z$, the set of $v$ such that $(u,v)\in \Omega$ corresponds to the set of all
\emph{greedy partitions} of $u$ into words in $\mathcal L(Z)$: each such word corresponds to a maximal string in $v$ of
the form $10\ldots0$.

In case $u\in A^\Z\setminus Z$, one may obtain a pinning sequence $v$ such that $(u,v)\in \Omega$
by a limit of greedy algorithms as follows:
for each $n$, let $k^{(n)}_0=-n$ and let $k^{(n)}_{i+1}$ be the smallest integer greater than $k^{(n)}_i$
such that $u_{k^{(n)}_i}\ldots u_{k^{(n)}_{i+1}}\not\in\mathcal L(Z)$ (where the sequence terminates if
there is no such $k^{(n)}_{i+1}$). Then define a sequence $(v^{(n)})$ by
$$
v^{(n)}_j=\begin{cases}1&\text{if $j\in \{k^{(n)}_i:i\ge 0\}$};\\0&\text{otherwise}.
\end{cases}
$$
Any subsequential limit $v$ of the $v^{(n)}$ sequences satisfies $(u,v)\in \Omega$.
In particular, $A^\Z$ is a factor of the shift on $\Omega$ by the projection onto the first coordinate.

Given an ergodic invariant measure $\mu$ on $A^\Z$, we now build a suitable lift to $\Omega$.
Denote by $\bm{\updelta}_p$ the Dirac measure supported on the point $p$.
If $\mu$ is supported on $Z$, then clearly $\bar\mu=\mu\times\bm{\updelta}_{\mathbf 0}$ is a suitable lift.
If not, let $u$ be a generic point of $A^\Z$ for $\mu$, and let $v$ be its pinning sequence so that
$(u,v)\in\Omega$. We then let $\bar\nu$ be a subsequential limit of the sequence
$\frac 1n(\bm{\updelta}_{(u,v)}+\ldots+\bm{\updelta}_{\bar\sigma^{n-1}(u,v)})$.
By the $\mu$-genericity of $u$, the projection of $\bar\nu$ onto the $A^\Z$ coordinate is $\mu$.
Since $\bar\nu$ may fail to be ergodic, we consider the ergodic components of $\bar\nu$.
By ergodicity of $\mu$, almost every ergodic component of $\bar\nu$ is supported on $\Omega$ and projects
onto $\mu$. We let $\bar\mu$ be any such ergodic component and call $\bar\mu$ a \emph{lift} of $\mu$
to $\Omega$.

If $\mu$ is an ergodic invariant measure supported on $A^\Z\setminus \Cl(Z)$, there is a word $w\not\in\mathcal
L(Z)$ such that $\mu([w])>0$. By ergodicity $\mu$-a.e.{} $u$ contains infinitely many copies of the word
$w$, so that if $(u,v)\in\Omega$, then over each occurrence of $w$ in $u$, there is at least one pin (1
in the corresponding pinning sequence, $v$). By the greedy property, a single 1 in $v$, together
with $u$ determines all subsequent terms of $v$. It follows that there are at most $|w|$ different $v$'s such that
$(u,v)\in \Omega$. In particular, the projection map from the pinning space to $A^{\Z}$ is
$\bar\mu$-almost surely finite-to-one. It follows that $h(\bar\mu)=h(\mu)$. Since $\bar\mu$ projects
to $\mu$, we see immediately that $\int f\circ \pi\,d\bar\mu=\int f\,d\mu$ for any Borel function $f$ on
$A^\Z$, where $\pi$ is the projection of $\Omega$ onto the first coordinate, $A^\Z$.

Let $P=\{(u,v)\in\Omega\colon v_0=1\}$ and let $\bar\phi_k(u,v)=\phi_k(u)$ for $k=1,\ldots,m$.
It follows from the above that to prove (\ref{eq:var_pr_ineq}) it suffices to show that for any
$(t_1,\ldots,t_m)\in (\alpha,\infty)^m$, and any ergodic measure
$\bar\mu$ on $\Omega$ such that $\bar\mu(P)>0$,
\begin{equation}\label{eq:measures_on_Omega_estimate}
  h(\bar\mu)+\int (t_1\bar\phi_1+\ldots+t_m\bar\phi_m)\,d\bar\mu< F(t_1,\ldots,t_m).
\end{equation}

Let $\bar\mu$ be an ergodic measure on $\Omega$ such that $\bar\mu(P)>0$. Let $\tau_P(u,v)=
\min\{i\ge 1: \bar\sigma^i(u,v)\in P\}$ be the first return time to $P$.
Let $\bar\sigma_P$ denote the induced map of $\bar\sigma$ on $P$ with invariant measure
$\bar\mu_P(\cdot)=\bar\mu(\cdot\cap P)/\bar\mu(P)$, i.e. $\bar\sigma_P(u,v)=\bar \sigma^{\tau(u,v)}(u,v)$.
By Abramov's formula \cite{Abramov59} the relation between the entropy of the measure $\bar\mu$ on
$(\Omega,\bar\sigma)$ and the entropy of the induced measure $\bar\mu_P$ on $(P,\bar\sigma_P)$ is
$$
h(\bar\mu_P)= \frac{1}{\bar\mu(P)}h(\bar\mu).
$$

We  introduce three countable partitions of $P$.
Let $\mathcal Q=\{Q_1,\ldots\}$ be the partition of $P$ according to the return time to $P$. Here
$$
Q_j=\{(u,v)\in P\,:\, \tau_P(u,v)=j \}.
$$
We let $\mathcal R$ to be a subpartition of $\mathcal Q$ according to the weights in each of the Sturmian
components. Precisely, given $(u,v)\in Q_j$ we write out the components of $u$ so that $(u,v)=(x,y^0,\ldots,y^m,v)$.
For each $j\in\N$ and a tuple $\bn=(n_0,\ldots,n_m)\in \Z^{m+1}$ we define
$$
R_{j,\bn}=\{(x,y^0,\ldots,y^m,v)\in Q_j\,:\,y^k_0+\ldots+y^k_{j-1}=n_k \text{ for } k=0,\ldots,m\}.
$$
We denote by $N_j$ the set of tuples $\bn$ for which the set $R_{j,\bn}$ is not empty. Then $\mathcal R$ is the partition $\{R_{j,\bn}\colon j\in\N,\ \bn\in N_j\}$.
Finally, let $\mathcal P$ denote the partition of $P$ in which each $Q_j$ is refined into cylinder sets
of length $j$. In particular, $\mathcal P$ is a generating partition under $\bar\sigma_P$; the partitions
$\mathcal Q$, $\mathcal R$ and $\mathcal P$ are successive refinements.
We introduce the notation
\begin{equation}\label{notation}
  q_j=\bar\mu_P(Q_j)\quad\text{and}\quad r_{j,\bn}=\bar\mu_P(R_{j,\bn}).
\end{equation}
We have a set of equalities which will be extensively used in what follows:
\begin{align}
\label{eq:q_j}  &\sum_{j=1}^\infty q_j=1 \\
\label{eq:r_j,n}   & \sum_{\bn\in N_j} r_{j,\bn}=q_j \\
\label{eq:Kac's_Lemma}  & \sum_{j=1}^\infty jq_j=1/\bar\mu(P)\quad\text{(Kac's lemma)}
\end{align}

First, we establish a connection between the elements of the partition $\mathcal R$ and subshifts forming $Z$.

\begin{lem}\label{lem:gammaest}
Suppose $(u,v)\in R_{j,\bn}$ with $\bn=(n_0,\ldots,n_m)$ and let $x=\pi_\beta(u)$ and
$y^k=\pi_k(u)$ for $k=0,\ldots,m$ so that $u=(x,y^0,\ldots,y^m)$.
Then there is $\bgamma\in S$ such that for each $k=0,\ldots,m$,
$\gamma_k\in (\frac{n_k-1}j,\frac{n_k+1}j)$,
$x_0\ldots x_{j-1}\in\mathcal L_j(X_{e^{\gamma_0}})$
and $y_0^k\ldots y_{j-1}^k\in \mathcal L_j(Y_{\gamma_k})$.
\end{lem}

\begin{proof}
Let $(u,v)\in R_{j,\bn}$, $x=\pi_\beta(u)$ and $y^k=\pi_k(u)$ for $k=0,1,\ldots,m$.
Since $u\in \mathcal L_j(Z)$, it follows that $u\in \mathcal L_j(Z_{\bgamma})$
for some $\bgamma=(\gamma_0,\ldots,\gamma_m)\in S$. This implies that
$x_0\ldots x_{j-1}\in \mathcal L_j(X_{e^{\gamma_0}})$ and
$y^k_0\ldots y^k_{j-1}\in\mathcal L_j(Y_{\gamma_k})$ for $0\le k\le m$. It remains to show that
$\gamma_k\in (\frac{n_k-1}j,\frac{n_k+1}j)$ where $\bn=(n_0,\ldots,n_m)$.

Since $(x,y^0,\ldots,y^m,v)\in R_{j,\bn}$ for each $k=0,\ldots,m$ we have
$n_k=y^k_0+\ldots+y^k_{j-1}$. On the other hand,
$y^k_0\ldots y^k_{j-1}\in\mathcal L_j(Y_{\gamma_k})$ implies by Lemma
\ref{lem:SturmChar} that $\lfloor j\gamma_k\rfloor\le n_k \le \lceil j\gamma_k\rceil$
so that for each $0\le k\le m$, $j\gamma_k\in(n_k-1,n_k+1)$ as required.
\end{proof}

Next we obtain an upper bound on the entropy of the measure $\bar\mu$ on $\Omega$ via the
entropy of the corresponding induced measure $\bar \mu_P$. We use the notation introduced in (\ref{notation}).

\begin{lem}\label{lem:entropy_est}
Suppose $\bar\mu$ is an ergodic invariant measure on $\Omega$ such that $\bar\mu(P)>0$. Then
$$
h(\bar\mu_P)\le c+6+(2L+2)m +(3m+6)\sum_{j=1}^\infty q_j\log j+\sum_{j=1}^\infty\sum_{\bn\in N_j}n_0r_{j,\bn}.
$$
\end{lem}

\begin{proof}
Recall that $\mathcal P$ is a generating partition for $\bar\sigma_P$ and hence $$h(\bar\mu_P)=\inf\frac1n H_{\bar\mu_P}\left(\bigvee_{i=0}^{n-1}\bar\sigma_P^{-i}(\mathcal P)\right),$$
where $H_{\bar\mu_P}(.)$ is the entropy of the partition with respect to the measure $\bar\mu_P$, see e.g. \cite{Ke}.
Since $H_{\bar\mu_P}(\bigvee_{i=0}^{n-1}\bar\sigma_P^{-i}(\mathcal P))\le n H_{\bar\mu_P}(\mathcal P)$, we obtain the simple bound $h(\bar\mu_P)\le H_{\bar\mu_P}(\mathcal P)$.
We then estimate $H_{\bar\mu_P}(\mathcal P)$ using conditional entropy: $$H_{\bar\mu_P}(\mathcal P)=
H_{\bar\mu_P}(\mathcal Q)+H_{\bar\mu_P}(\mathcal R|\mathcal Q)+
H_{\bar\mu_P}(\mathcal P|\mathcal R).$$

We have $H_{\bar\mu_P}(\mathcal Q)=\sum_{j=1}^\infty -q_j\log q_j$, which we separate into two parts as
$$\sum_{j=1}^\infty -q_j\log q_j=-\sum_{q_j<1/j^2}q_j\log q_j-\sum_{q_j\ge 1/j^2}q_j\log q_j.$$
The first term is at most $\frac 1e-\sum_{j=2}^\infty \frac{1}{j^2}\log \frac{1}{j^2}$
(which we obtained using the fact that $-t\log t$ is increasing on $[0,\frac 1e]$ and bounded
above by $\frac 1e$ on $[0,1]$). The second term is bounded above by $2\sum_{j=1}^\infty q_j\log j$, so that
$$
H_{\bar\mu_P}(\mathcal Q)\le 3+2\sum_{j=1}^\infty q_j\log j.
$$

We now turn to $H_{\bar\mu_P}(\mathcal R|\mathcal Q)$, which is given by
$$
H_{\bar\mu_P}(\mathcal R|\mathcal Q)=\sum_{Q\in\mathcal Q}\bar\mu_P(Q)
\left(-\sum_{R\in\mathcal R}\frac{\bar\mu(R\cap Q)}{\bar\mu_P(Q)}
\log \frac{\bar\mu(R\cap Q)}{\bar\mu_P(Q)}\right).
$$
We bound the term in parentheses by the logarithm of the number of elements in
$\mathcal R$ into which the set $Q$ is partitioned. Recall that $\mathcal Q=\{Q_j\}_{j\in\N}$ and each
$Q_j=\bigcup_{\bn\in N_j}R_{j,\bn}$. If $\bn=(n_0,\ldots,n_m)\in N_j$ then $R_{j,\bn}$
is not empty and hence by Lemma \ref{lem:gammaest} there is
$\bgamma=(\gamma_0,\ldots,\gamma_m)\in S$ such that
$\lfloor j\gamma_k\rfloor\le n_k \le \lceil j\gamma_k\rceil$ for each $k=0,\ldots,m$.
Since $S\subset [b,c]\times [-L,L]^m$, there are at most
$\lceil c\rceil j\cdot(2\lceil L\rceil+1)^mj^m$ tuples $\bn$ in $N_j$. Now using
(\ref{notation}), (\ref{eq:q_j}) and the fact that $\log \lceil c\rceil \le c$ we obtain
\begin{align*}
  H_{\bar\mu_P}(\mathcal R|\mathcal Q)& \le \sum_{j=1}^{\infty}
  q_j\log\left( \lceil c\rceil(2\lceil L\rceil+1)^mj^{m+1} \right)\\
  & \le  \sum_{j=1}^{\infty}q_j(c+(2L+1)m+(m+1)\log j)\\
  & = c+(2L+1)m+(m+1)\sum_{j=1}^{\infty}q_j\log j
\end{align*}

Finally, we estimate $H_{\bar\mu_P}(\mathcal P|\mathcal R)$. Similar to the above,
we use a crude bound via the number of $j$-cylinders forming each $R_{j,\bn}\in \mathcal R$.
Recall that each element of $\mathcal P$ is a cylinder set generated by an element of
$\mathcal L(Z)$. We separately estimate the number of projections of those words forming
$R_{j,\bn}$ onto each coordinate. Suppose $(u,v)\in R_{j,\bn}$. Write $\bn=(n_0,\ldots,n_m)$,
$x=\pi_\beta(u)$ and $y^k=\pi_k(u)$ for $0\le k\le m$.

There are at most $j(j+1)\le 2j^2$ choices for $y^k_0\ldots y^k_{j-1}$ by Lemma
\ref{lem:Sturm}, since each such Sturmian word must have the same weight $n_k$.
By Lemma \ref{lem:gammaest}, and using the fact that $X_\beta\subset X_{\beta'}$
if $\beta<\beta'$, we see $x_0\ldots x_{j-1}\in X_{e^{(n_0+1)/j}}$. By Lemma
\ref{lem:betacount}, the number of such choices of $x_0\ldots x_{j-1}$ is at most
\begin{align*}
\frac{e^{(n_0+1)/j}}{e^{(n_0+1)/j}-1}e^{n_0+1}&=
\frac{1}{1-e^{-(n_0+1)/j}}e^{n_0+1}\\
&\le \frac{1}{1-e^{-1/j}}e^{n_0+1}\le je^2e^{n_0},
\end{align*}
using the fact that $1/(1-e^{-1/j})\le ej$.

Multiplying the estimates, we see that the number of $j$-cylinders making up each
$R_{j,\bn}$ is at most $(2j^2)^{m+1} je^2e^{n_0}$. Therefore,
\begin{align*}
  H_{\bar\mu_P}(\mathcal P|\mathcal R)
  & \le \sum_{R\in\mathcal R}\bar \mu_P(R)\log\left(2^{m+1}j^{2m+3}e^2e^{n_0}\right) \\
  & = \sum_{j=1}^{\infty}\sum_{\bn\in N_j}r_{j,\bn}\big((m+1)\log 2 +2+(2m+3)\log j+n_0\big) \\
  & \le \sum_{j=1}^{\infty} q_j(m +3)+(2m+3)\sum_{j=1}^{\infty}q_j\log j
  +\sum_{j=1}^{\infty}\sum_{\bn\in N_j}n_0r_{j,\bn}\\
  & = m+3+(2m+3)\sum_{j=1}^{\infty}q_j\log j+\sum_{j=1}^{\infty}\sum_{\bn\in N_j}n_0r_{j,\bn},
\end{align*}
where we used (\ref{eq:r_j,n}) in the second line and (\ref{eq:q_j}) in the third.
Combining the above estimates we establish that
$$
H_{\bar\mu_P}(\mathcal P)\le c+6+(2L+2)m +(3m+6)\sum_{j=1}^\infty q_j\log j
+\sum_{j=1}^\infty\sum_{\bn\in N_j}n_0r_{j,\bn}.
$$
\end{proof}

Lastly, we estimate $\int\phi_k\,d\mu$ for $0\le k \le m$. We define a version of
$\bar\phi_k$ on the induced system by
$$
\bar\phi^P_k(u,v)=\sum_{i=0}^{\tau_P(u,v)-1}\phi_k(\sigma^i u),
$$
so that $\int\bar\phi_k^P\,d\bar\mu_P=\frac{1}{\bar\mu(P)}\int\bar\phi_k\,d\bar\mu=\frac{1}{\bar\mu(P)}
\int\phi_k\,d\mu$. We continue using the notation from (\ref{notation}).

\begin{lem}\label{lem:int_est}
Suppose $\bar\mu$ is an ergodic invariant measure on $\Omega$ such that $\bar\mu(P)>0$.
Then for $k=1,\ldots,m$ we have
$$
\int \bar\phi_k^P\,d\bar\mu_P \le 3+\sum_{j=1}^\infty\sum_{\bn\in N_j} n_k r_{j,\bn}
-\sum_{j=1}^\infty jq_j\delta_j,
$$
where $\delta_j$ is defined as in (\ref{eq:delta_j}).
\end{lem}

\begin{proof}
We fix $1\le k\le m$ and estimate $\bar\phi_k^P(u,v)$ for $(u,v)\in R_{j,\bn}$.
From the construction of $R_{j,\bn}$ we know that
$u_0\ldots u_j\not\in\mathcal L(Z)$. It follows that whenever $\bgamma=(\gamma_0,\ldots,\gamma_m)\in S$
$$\phi_{k,\bgamma}(\sigma^i u)\le \gamma_k-\delta_j$$
for each $i=0,\ldots,j-1$.

On the other hand, using Lemma \ref{lem:gammaest} we can find $\bgamma'\in S$
such that $u_0\ldots u_{j-1}\in\mathcal L_j(Z_{\bgamma'})$ and
$\gamma'_k\in\left(\frac{n_k-1}{j},\frac{n_k+1}{j}\right)$. Consider any other
$\bgamma=(\gamma_0,\ldots,\gamma_m)\in S$. We have the following possibilities
for the location of $\gamma_k$ with respect to $\gamma_k'$: $\gamma_k$ is
less than $\gamma_k'+\frac{2}{j}$; $\gamma_k$ is in one of the intervals
$\left[\gamma'_k+\frac{2}{\ell},\gamma_k'+\frac{2}{\ell-1}\right)$ where
$\ell\in\{2,\ldots,j\}$; or $\gamma_k$ is at least $\gamma_k'+2$.

If $\gamma_k <\gamma_k'+\frac{2}{j}$ then  we see that
$\gamma_k<\frac{n_k}{j}+\frac{3}{j}$ since
$\gamma'_k\in\left(\frac{n_k-1}{j},\frac{n_k+1}{j}\right)$. Therefore,
$$
\phi_{k,\bgamma}(\sigma^i u)\le \gamma_k-\delta_j<\frac{n_k}{j}+\frac{3}{j}-\delta_j.
$$

Now suppose that $\gamma_k\in \left[\gamma'_k+\frac{2}{\ell},\gamma_k'+\frac{2}{\ell-1}\right)$
for some $\ell\in \{2,\ldots,j\}$. We claim that $\mathcal L_\ell(Y_{\gamma_k})$ and $\mathcal L_\ell(Y_{\gamma_k'})$ are disjoint.
To see this, note that any element of $\mathcal L_\ell(Y_{\gamma_k})$ has weight at least $\lfloor \ell\gamma_k\rfloor$
while any element of $\mathcal L_\ell(Y_{\gamma_k'})$ has weight at most $\lceil \ell\gamma_k'\rceil$.
Since $\gamma_k\ge \gamma_k'+\frac 2\ell$. it follows that $\ell\gamma_k\ge \ell\gamma_k'+2$, so that
$\lfloor \ell\gamma_k\rfloor > \lceil \ell\gamma_k'\rceil$.
It follows that
$\phi_{k,\bgamma}(\sigma^iu)\le\gamma_k-\delta_{\ell-1}$ for each $i=0,\ldots,j-1$.
Using that $\gamma_k\in \left[\gamma'_k+\frac{2}{\ell},\gamma_k'+\frac{2}{\ell-1}\right)$
and that $\left(\frac{2}{j}-\delta_j\right)$ is an increasing sequence we obtain
$$
\phi_{k,\bgamma}(\sigma^iu)\le\gamma_k-\delta_{\ell-1}\le \gamma'_k+\frac{2}{\ell-1}-\delta_{\ell-1}<\frac{n_k+1}{j}+\frac{2}{j}-\delta_j.
$$

Finally, let $\gamma_k\ge \gamma'_k+2$. Since in this case
$\lfloor\ell\gamma_k+a\rfloor>\lfloor \ell \gamma'_k+a' \rfloor$ for all $a,a'\in [0,1)$
we see that $y^k_i\notin\mathcal L_1(Y_{\gamma_k})$ for $i=0,\ldots,j-1$. Hence,
$$
\phi_{k,\bgamma}(\sigma^iu)\le\gamma_k-\delta_{0}\le 2L+\gamma'_k-\delta_0< \frac{n_k+1}{j}-\delta_j,
$$
since $\delta_0-2L\ge \delta_1>\delta_j$ be definition.

We have shown that for all $\bgamma\in S$, for all $(u,v)\in R_{j,\bn}$ and all
$i=0,\ldots,j-1$ we have $\phi_{k,\bgamma}(\sigma^iu)\le \frac{n_k}{j}+\frac{3}{j}-\delta_j$.
Hence, $\phi_k(\sigma^i u)=\sup_{\bgamma\in S}\phi_{k,\bgamma}(\sigma^iu)
\le \frac{n_k}{j}+\frac{3}{j}-\delta_j $ and
$$
\bar\phi^P_k(u,v)=\sum_{i=0}^{j-1}\phi_k(\sigma^i u)
\le j\left(\frac{n_k}{j}+\frac{3}{j}-\delta_j \right)=n_k+3-j\delta_j.
$$
Now integrating and applying (\ref{eq:r_j,n}), we see
\begin{align*}
  \int \bar\phi_k^P\,d\bar\mu_P & \le \sum_{j=1}^\infty\sum_{\bn\in N_j}(n_k+3-j\delta_j)\bar\mu_P(R_{j,\bn})\\
  & = 3+\sum_{j=1}^\infty\sum_{\bn\in N_j} n_k r_{j,\bn}-\sum_{j=1}^\infty jq_j\delta_j,
\end{align*}
as required.
\end{proof}

We are now ready to establish (\ref{eq:var_pr_ineq}). Fix the values of the parameters
$(t_1,...,t_m)\in (\alpha,\infty)^m$. Suppose $\mu$ is an ergodic $\sigma$-invariant measure on $A^\Z$
whose support is not contained in $\Cl(Z)$. Then its lift $\bar \mu$ is an ergodic $\bar\sigma$-invariant measure
on $\Omega$ such that $\bar \mu(P)>0$ and we can induce on $P$. Combining the estimates in
Lemma \ref{lem:entropy_est} and Lemma \ref{lem:int_est}, we see that
\begin{equation}\label{eq:star}
\begin{split}
  &h(\bar\mu_P)+\int(t_1\bar\phi_1^P+\ldots+t_m\bar\phi_m^P)\,d\bar\mu_P  \\
  &\le c+6+(2L+2)m +(3m+6)\sum_{j=1}^\infty q_j\log j+\sum_{j=1}^\infty\sum_{\bn\in N_j}n_0r_{j,\bn}  \\
  &+\sum_{j=1}^\infty\sum_{\bn\in N_j} (t_1n_1+\ldots+t_m n_m) r_{j,\bn}+(t_1+\ldots+t_m)\left(3
  -\sum_{j=1}^\infty jq_j\delta_j \right).
  \end{split}
\end{equation}

We estimate the terms containing $r_{j,\bn}$ first. Let $j \in \N$ and let $\bn=(n_0,...,n_m) \in N_j$.
Since the set $R_{j,\bn}$ is not empty, by Lemma \ref{lem:gammaest} for each such $j$ and $\bn$ we can find some
$\bgamma=(\gamma_0,...,\gamma_m)\in S$ satisfying $n_k<j\gamma_k+1$ for $k=0,\ldots,m$. Therefore,
$$
n_0+t_1n_1+\ldots+t_mn_m\le j(\gamma_0+t_1\gamma_1+\ldots+t_m\gamma_n)+1+t_1+\ldots+t_m.
$$
Since $\bgamma\in S$, Lemma \ref{lem:ConvAnal} implies that
$\gamma_0+t_1\gamma_1+\ldots+t_m\gamma_n\le F(t_1,\ldots,t_m)$. Writing $\bt=(t_1,...,t_m)$ and using (\ref{eq:r_j,n}) we get
\begin{align*}
\sum_{\bn\in N_j} (n_0+t_1 n_1\ldots+t_m n_m) r_{j,\bn}&\le \sum_{\bn\in N_j} (jF(\bt)+1+t_1+\ldots+t_m) r_{j,\bn}\\
  &=[jF(\bt)+1+t_1+\ldots+t_m]q_j.
\end{align*}
Recall from (\ref{eq:q_j}) and (\ref{eq:Kac's_Lemma}) that $\sum_j q_j=1$ and $\sum_j jq_j=\frac{1}{\bar\mu (P)}$. Hence, summing over $j$ gives
\begin{align*}
  \sum_{j=1}^\infty\sum_{\bn\in N_j} (n_0+t_1 n_1\ldots+t_m n_m) r_{j,\bn}&\le \sum_{j=1}^\infty[jF(\bt)+1+t_1+\ldots+t_m]q_j \\
    &=\frac{F(\bt)}{\bar{\mu}(P)}+1+t_1+\ldots+t_m.
\end{align*}
Substituting the bound we just obtained for the terms containing $r_{j,\bn}$ into (\ref{eq:star}) and applying $\sum q_j=1$ in the last line we get
\begin{align*}
&h(\bar\mu_P)+\int(t_1\bar\phi_1^P+\ldots+t_m\bar\phi_m^P)\,d\bar\mu_P\\
&\le(c+2L+9)m+9m\sum_{j=1}^\infty q_j\log j
+(t_1+\ldots+t_m)\left(4-\sum_{j=1}^\infty jq_j\delta_j\right)\\
&\hspace{11.1cm}+\frac{F(\bt)}{\bar\mu(P)}\\
&=\sum_{j=1}^\infty q_j\big[(c+2L+9)m+9m\log j-(t_1+\ldots+t_m)(j\delta_j-4)\big]+\frac{F(\bt)}{\bar\mu(P)}.
\end{align*}
Since
$$
j\delta_j>\frac{c+2L+9+9\log j}{\alpha}+4
$$
and $t_1+...+t_m>m\alpha$ we observe that the bracketed expression is negative.
This gives
$$
h(\bar\mu_P)+\int(t_1\bar\phi_1^P+\ldots+t_m\bar\phi_m^P)\,d\bar\mu_P
<\frac{F(t_1,\ldots,t_m)}{\bar\mu(P)}.
$$

As mentioned above $h(\mu)=h(\bar\mu)$; Abramov's formula then implies that
$h(\mu)=\bar\mu(P)h(\bar\mu_P)$. Also, $\int\phi_k\,d\mu=\bar\mu(P)\int\bar\phi_k^P\,d\bar\mu_P$
so that
$$
h(\mu)+\int (t_1\phi_1+\ldots+t_m\phi_m)\,d\mu< F(t_1,\ldots,t_m),
$$
as required. This completes the proof of Theorem \ref{thm:main}.

\begin{rem}\label{rem:full shift}
Note that allowing $\alpha$ to be zero in the statement of the theorem makes it false. Any pressure function
on a full shift $\Sigma$ intercepts the vertical axis at $h_{\rm top} (\Sigma)$. Hence, if the value of $F$ at the origin
is not equal to the logarithm of an integer greater than 1, $F$ cannot be a pressure function on $[0,\infty)^m$ for any full shift.
However, for any $\alpha>0$  we can still match $F$ to a pressure function on $(\alpha,\infty)^m$.

The spirit of the proof is that each potential $\phi_j$ is defined to be $\gamma_j$ minus a ``cost". If $\phi_j$ is multiplied by a variable $t_j$, we require that the scaled cost be sufficiently large to ensure that invariant measures not supported on $Z$ do not achieve the desired pressure. This only works if $t_j$ is bounded below. This constraint manifests itself in the proof above where $\alpha$ appears in the denominator of (\ref{eq:delta_j}).

\end{rem}

\section{One-parameter pressure function}\label{sec:OneParameterPressure}

Of particular interest is a one-parameter pressure function $t\mapsto P(t\phi)$ since then $t$ can be
interpreted as the inverse temperature of the system. It follows immediately from the variational principle
that the pressure function is convex, Lipschitz and asymptotically linear at infinity. As a consequence of
Theorem \ref{thm:main} we see that these are the only restrictions. Furthermore, in the one-parameter
situation boundedness of the vertical axis intercepts of the supporting lines implies both the Lipschitz
condition and the existence of a slant asymptote. Indeed, we have

\begin{lem}\label{lem:main_1D}
  Let $\alpha>0$ and let $f(t)$ be a convex function on $(\alpha,\infty)$ such that the support lines to $f$ at
each $t\in(\alpha,\infty)$ have vertical axis intercepts in a closed interval $[b,c]\subset [0,\infty)$.
Then $f(t)$ is Lipschitz and has a slant asymptote.
\end{lem}
\begin{proof}
We first show that the derivatives are uniformly bounded above. Note that for a single-variable function
$f$ its subdifferential at $t\in(\alpha,\infty)$ is the interval
$\partial f(t)=[f'(t^-),f'(t^+)]$, where $f'(t^-)$ and $f'(t^+)$ denote the left and right derivatives of $f$ at $t$ respectively.
As in the multi-variable case, the subdifferential of $f$ is characterized
by the property that $v\in\partial f(t)$ if and only if $f(t)+v(s-t)\le f(s)$ for all $s\in(\alpha,\infty)$.
Given $t\in(\alpha,\infty)$ and $v\in \partial f(t)$, the \emph{intercept} of the sub-tangent line
with slope $v$ is $f(t)-vt$, which is the intercept of the sub-tangent line $\ell(s)=f(t)+v(s-t)$ with the vertical axis.

Fix $\beta>\alpha$. Let $t\in(\alpha,\infty)$ be arbitrary and let $v\in\partial f(t)$.
Let $\iota\in [b,c]$ be the corresponding intercept.
Then $f(\beta)\ge \iota+v\beta\ge b+v\beta$.
It follows that $v\le (f(\beta)-b)/\beta$, giving a uniform upper bound on derivatives of $f$. Moreover, since $t$ was arbitrary we see that
$$
f(\beta)\ge b+\beta\sup_{t\in(\alpha,\infty)}\partial f(t)
$$
and this inequality holds for all $\beta>\alpha$.

Denote $\bar v=\sup_{t\in(\alpha,\infty)}\partial f(t)$ and assume that $b$ is the greatest lower bound of the vertical intercepts. We claim that $f(t)-(\bar v t+b)$ approaches 0.
Let $\epsilon>0$ and let $t_0$ be such that there is a supporting line $vt+\iota$ to $f(t)$,
touching at $t_0$ and satisfying $\iota<b+\epsilon$.
Then $b+\epsilon>\iota=f(t_0)-v t_0$ so that $f(t_0)\le v t_0+b+\epsilon\le \bar v t_0+b+\epsilon$. By the convexity of $f$, the supporting line touching at any $t>t_0$ has its vertical axis intercept below or equal to $\iota$ and hence the same inequality holds for all $t\ge t_0$.

For a lower bound on the derivatives, consider a sub-tangent line at $t$ with slope $v$ and
intercept $\iota$, so that $f(t)=\iota+vt$. In particular, we see $v\ge (f(t)-c)/t$.
Since $\partial f(s)\le \partial f(t)$ whenever $s\le t$, it suffices to show that
$\lim_{s\to\alpha^+}\partial f(s)$ is finite.
By convexity, $f(\alpha^+)$ exists and is at least $f(\beta)-u(\beta-\alpha)$ where
$u\in\partial f(\beta)$, so that $f(\alpha^+)\in(-\infty,\infty]$. Hence the
inequality above shows that $\partial f(\alpha^+)\ge (f(\alpha^+)-c)/\alpha$
giving the required lower bound.
\end{proof}

\begin{cor}\label{cor:main_1D}
Let $\alpha>0$ and let $f(t)$ be a convex function on $(\alpha,\infty)$ such that the support lines to $f$ at
each $t\in(\alpha,\infty)$ have vertical axis intercepts in a closed interval $[b,c]\subset [0,\infty)$.
Then there exists a full shift on a finite alphabet and a continuous potential $\phi$ such that
$P_{\rm top}(t\phi)=f(t)$ for all $t\in (\alpha,\infty)$.
\end{cor}

\begin{proof}
  Note that $f(t)$ is Lipschitz by the above lemma and then apply Theorem \ref{thm:main} with $m=1$.
\end{proof}

We point out that contrary to the one-parameter situation, the multi-parameter pressure function is no
longer asymptotically linear. Furthermore, in higher dimensions the fact that the intercepts are bounded
does not imply that the convex function is Lipschitz. We give an example to illustrate this.

\begin{example}
This is an example of a convex non-Lipschitz function $F$ such that the set of vertical-axis intercepts of all
the support planes to the graph of $F$ is bounded.
   \end{example}
\begin{proof}
For $(t_1,t_2)\in (0,\infty)^2$, set
$$
F(t_1,t_2)=\sup_{s\in[0,\infty)} (t_1s-t_2s^2).
$$
Notice that for a fixed $s$, $t_1s-t_2s^2$ is a linear function of $(t_1,t_2)$ so that $F$ is convex.
For fixed $(t_1,t_2)\in(0,\infty)^2$, $$t_1s-t_2s^2=\frac{t_1^2}{4t_2}-t_2\left(s-\frac{t_1}{2t_2}\right)^2,$$
so that $F(t_1,t_2)=\frac{t_1^2}{4t_2}$. Since $F(t_1,t_2)$ is the supremum of a collection of linear functions of $(t_1,t_2)$,
it is convex. Since each of the linear functions in the collection has intercept 0, the collection
of intercepts is bounded. However $F$ clearly fails to be Lipschitz.
\end{proof}

We finish this section with an application of our result to describe feasible occurrences of first-order
phase transitions. Let $\alpha>0$ and let $(z_j)$ be an arbitrary (possibly finite) sequence
of terms in $(\alpha,\infty)$. Let $S=\{z_j: j=1,2,...\}$. We define a function $\phi$ as follows.
First let $g\colon (\alpha,\infty)\to\R$ be given by
$$
g(s)=\sum_{\{j\colon z_j\le s\}}\frac{\alpha}{2^jz_j^2}.
$$
Then define
$$
f(t)=3+\int_0^t g(s)\,ds.
$$
Notice that since $\alpha/(2^jz_j^2)$ are summable, $g$ is continuous everywhere except
on $S$, where it jumps upwards, guaranteeing that $f$ is differentiable precisely on
$(\alpha,\infty)\setminus S$.

We claim that $f$ satisfies the hypotheses of Corollary \ref{cor:main_1D}. The vertical axis intercept of the
support line at $t\in(\alpha,\infty)$ (or the support line with the largest gradient if $t\in S$)
is given by
\begin{align*}
f(t)-tg(t)&=3-\int_0^t \big(g(t)-g(s)\big)\,ds\\
&=3-\int_0^t \sum_{\{j\colon s<z_j\le t\}}\frac \alpha{2^jz_j^2}\,ds\\
&\ge 3-\int_0^t \sum_j \frac{\alpha}{2^j\max(\alpha,s)^2}\,ds\\
&=3-\int_0^t \frac{\alpha}{\max(\alpha,s)^2}\,ds\\
&=3-\int_0^\alpha \frac 1\alpha\,ds-\int_\alpha^t\frac{\alpha}{s^2}\,ds\ge 1.
\end{align*}
The upper bound $f(t)-tg(t)\le3$ follows from the first equality of the above display equation since $g(s)$ is non-decreasing.

The theorem shows that we are able to construct a potential $\phi$ whose pressure function
has an arbitrary countable collection of first order phase transitions.

\section{Cardinality of equilibrium states}\label{sec:Cardinality}

In this section, we briefly outline a strategy for showing that not only is one free to specify
the pressure function, but there is also a lot of freedom in controlling the cardinality of
the set of ergodic equilibrium states. In particular, we prove Theorem \ref{thm:main+card}.

In the proof of Theorem \ref{thm:main}, for each function $F\colon (\alpha,\infty)^m\to\R$
satisfying the conditions of the theorem, we constructed
a full shift and a family of potentials $(\phi_i)_{i=1}^m$ such that
$F(t_1,\ldots,t_m)=P(t_1\phi_1+\ldots+t_m\phi_m)$ for each $\mathbf t\in (\alpha,\infty)^m$.
It is natural to ask about the cardinality of the set of equilibrium states for
these $t_1\phi_1+\ldots+t_m\phi_m$. The proof establishes that the ergodic equilibrium states for
$t_1\phi_1+\ldots+t_m\phi_m$ are precisely the measures of maximal entropy supported on
the $Z_\bgamma$ such that $\bgamma\in\partial F(\mathbf t)$.

For instance, if $\mathbf t$ is not a point
of differentiability for $F$, then there are multiple (uncountably many) ergodic equilibrium states
for $t_1\phi_1+\ldots+t_m\phi_m$. For $\mathbf t$ that are differentiability points of $F$,
there is exactly one element $\bgamma$ of $\partial F(\mathbf t)$. However, the space $Z_\bgamma=
X_{e^{\gamma_0}}\times Y_{\gamma_0}\times \ldots\times Y_{\gamma_m}$ may still support
uncountably many measures of maximal entropy if there is a rational relationship between
$\gamma_0,\gamma_1,\ldots,\gamma_m$ (more specifically, if there exists a non-trivial integer
combination of irrational $\gamma_j$'s taking an integer value).

For one-parameter pressure functions we can modify our construction slightly and obtain uniqueness of the
equilibrium states everywhere except for the points of non-differentiability. The main difference is that
instead of parameterizing
by the supporting hyperplanes, $S$, we can parameterize simply by the intercept of the support line with the
vertical axis. It is necessary to use a single Sturmian component
rather than the two Sturmian components in order to avoid possible rational dependencies between
gammas as described above.

When $m=1$ in Theorem \ref{thm:main} we have $Z_{\bgamma}=X_{e^{\gamma_0}}\times Y_{\gamma_0}\times Y_{\gamma_1}$.
We express $\gamma_1$ as a function of $\gamma_0$, which makes the factor $Y_{\gamma_1}$ redundant.
Recall from Section \ref{sec:beta-shifts} that $X_{e^{\gamma_0}}$ has a unique measure
of maximal entropy which is weak-mixing and Bernoulli.
Also $Y_{\gamma_0}$ is uniquely ergodic, so that $X_{e^{\gamma_0}}\times Y_{\gamma_0}$ supports a
unique measure of maximal entropy \cite{Fur}.
This measure is then the only equilibrium state of $t\phi$ in the case when the pressure function is
differentiable at $t$ and $\gamma_0$ is the vertical intercept of its tangent line at $t$.

To make the above precise, we fix a convex function $f(t)$ on $(\alpha,\infty)$ such that the support lines to $f$ at
each $t\in(\alpha,\infty)$ have vertical axis intercepts in a closed interval $[b,c]\subset [0,\infty)$.
For $\gamma\in [b,c]$ we define the function
$$
s(\gamma)=\sup\{\text{v}\colon \gamma+t\text{v}\le f(t)\text{ for $t\in(\alpha,\infty)$}\}.
$$
We show that the function $s(\gamma)$ is non-increasing and Lipschitz with Lipschitz constant $\frac 1\alpha$.
Let $\gamma<\gamma'$ and let $\tv$ be such that $\gamma'+t\tv\le f(t)$ for all $t\in(\alpha,\infty)$,
then $\gamma+t\tv\le f(t)$ for all $t\in(\alpha,\infty)$, so that $s(\gamma)\ge s(\gamma')$.
Next, observe that if $\gamma+t\tv\le f(t)$ for all $t\in(\alpha,\infty)$, then $\gamma'+t(\tv-\frac{\gamma'-\gamma}{\alpha})
\le \gamma+t\tv\le f(t)$ for all $t\in(\alpha,\infty)$, so that $s(\gamma')\ge \tv-\frac{\gamma'-\gamma}\alpha$ and
$|s(\gamma')-s(\gamma)|\le \frac{|\gamma'-\gamma|}\alpha$ as required.

It is easy to verify that for each $t\in (\alpha,\infty)$
$$
f(t)=\sup_{\gamma\in [b,c]}(\gamma+s(\gamma)t).
$$
Hence, we let the alphabet $A=\{0,1,...,\lfloor e^c\rfloor\}\times\{\lfloor b\rfloor,...,\lceil c\rceil\}$ and for $z\in A^{\Z}$ define
$$
\phi_\gamma(z)=s(\gamma)-\delta_{j_\gamma(z)}\quad\text{and}\quad \phi(z)=\sup_{\gamma\in [b,c]}\phi_\gamma(z),
$$
where $j_\gamma(z)$ and $\delta_{j_\gamma(z)}$ are as in (\ref{eq:delta_j}) with $Z_\gamma=X_{e^\gamma}\times Y_\gamma$.
The uniform equicontinuity of the family $\{\phi_\gamma\colon \gamma\in[b,c]\}$ ensures that $\phi$ is continuous.

We still need to confirm that $\phi(z)=s(\gamma)$ whenever $z=(x,y)\in Z_\gamma$. Since $s(\gamma)$
is non-increasing, $\phi_{\gamma'}\le s(\gamma)$ for $\gamma'\ge\gamma$. For $\gamma'<\gamma$ we
choose $j=\lceil 1/(\gamma-\gamma')\rceil$ and by looking at the weight of the word $y_{-j}...y_{j-1}$ conclude
that it is not in $\mathcal L_{2j}(Y_\gamma')$. It follows that $\phi_{\gamma'}(z)\le s(\gamma')-\delta_j$.
Since $j\le \frac{1+c}{\gamma-\gamma'}$ and $s(\gamma)$ is Lipschitz with constant $\frac{1}{\alpha}$, we see that
$$
\phi_{\gamma'}(z)\le s(\gamma')-\delta_j\le s(\gamma')-\frac{1+c}{\alpha j}\le s(\gamma')-\frac{\gamma-\gamma'}{\alpha}\le s(\gamma).
$$

The rest is a verbatim repetition of the proof of Theorem \ref{thm:main} with $\gamma=\gamma_0$ and $m=0$.
The only minor adjustment is in Lemma \ref{lem:int_est}, where the integral estimate becomes
$$
\int \bar\phi^P\,d\bar\mu_P \le \tfrac3\alpha+\sum_{j,n} jr_{j,n}s\big(\tfrac nj\big)-\sum_jjq_j\delta_j.
$$
The reason is that the value of the potential $\phi$ on each $R_{j,n}$ is approximately $s\big(\tfrac nj\big)$,
which follows from Lemma \ref{lem:gammaest} and the fact that $s(\gamma)$ is Lipschitz.

We have established the initial construction where the potential $t\phi$ has a unique equilibrium state for each
$t$ where $f(t)$ is differentiable. Now we are in position to add equilibrium states to $t\phi$ at various points
$t\in (\alpha,\infty)$ as we see fit. The key idea is to replace the sets $Z_\gamma$ which support the
equilibrium states for $t\phi$ by
$$
Z_\gamma=X_{e^\gamma}\times Y_\gamma\times D_\gamma
$$
where $D_\gamma$ is a \emph{decoration factor}. We now describe a set of desired properties of the family $\{D_\gamma\}$ and also outline one possible construction of such a family. Suppose for $t\in (\alpha,\infty)$ we would like the potential
$t\phi$ to have precisely $N(t)$ ergodic equilibrium states. Then we impose the following conditions on $\{D_\gamma:\gamma\in [b,c]\}$:
\begin{enumerate}[(i)]
\item \label{it:nmeas}
for $\gamma\in \partial f(t)$ the subshift $Z_\gamma$ supports exactly $N(t)$ ergodic measures of maximal entropy;
\item\label{it:zeroent}
$h_\text{top}(\text{Cl}(\bigcup_{\gamma\in [b,c]} D_\gamma))=0$;
\item\label{it:invmeas}
For any $\gamma\in (b,c)$, any invariant measure supported on the set
$\bigcap_{\varepsilon>0}\text{Cl}\big(\bigcup_{\gamma'\in (\gamma-\varepsilon,\gamma+\varepsilon)} Z_{\gamma'}\big)$ is
supported on $Z_\gamma$.
\end{enumerate}
Condition \ref{it:zeroent} gives an additional term in Lemma \ref{lem:entropy_est},
which has to be compensated for in the definition of $\delta_j$. The fact that the additional factor
has zero topological entropy ensures that $\delta_j$ still converges to 0.
Condition \ref{it:invmeas} is a mild extension of Lemma \ref{lem:measures on Z} to this context. 

Theorem \ref{thm:main+card} is an application of our technique, which illustrates the flexibility of cardinalities of
equilibrium measures. Note that the first implication of the statement of the theorem follows from Corollary~\ref{cor:main_1D}.
To prove the second implication, we parameterize by the intercept of the tangent line with the
vertical axis as outline above: for each intercept, $\gamma$, the line
$\gamma+s(\gamma) t$ is tangent to $f(t)$. We let the point of tangency be $\tau(\gamma)$. The function
$\tau$ is a homeomorphism from $(b,c)$, the interior of the set of intercepts, to $(\alpha,\infty)$.

We use the following choices:
$$
D_\gamma=\begin{cases}
\{\bar i\colon 1\le i\le N(\tau(\gamma))\}
&\text{if $N(\tau(\gamma))$ is finite;}\\
\text{Cl}\left(\bigcup_{t\in [1,\ell]}Y_t\right)
&\text{if $N(\tau(\gamma))=\infty$,}
\end{cases}
$$
where $\bar i$ denotes the fixed point $\ldots iii\cdot iii\ldots$ of the full shift on $\ell$ symbols.

As we pointed out before, $X_{e^\gamma}\times Y_\gamma$ supports a unique measure of maximal entropy.
If $N(\tau(\gamma))=k$, then $D_\gamma$ consists of $k$ fixed points, so that it is evident that $Z_\gamma$ supports exactly
$k$ ergodic measures of maximal entropy. If $N(\tau(\gamma))=\infty$, then $D_\gamma$ supports uncountably
many ergodic measures of maximal entropy, and so does $Z_\gamma$. This establishes condition \ref{it:nmeas}.

Condition \ref{it:zeroent} follows from a
theorem of Mignosi \cite{Mignosi}; or from Lemma \ref{lem:Sturm}. (Lemma
\ref{lem:Sturm} implies that the number of words of length $n$ is at most $(\ell n+1)n(n+1)$).

To establish  condition \ref{it:invmeas}, notice that the upper semi-continuity of $N(t)$ ensures
that for all $\gamma$,
\begin{equation}\label{eq:suppusc}
\bigcap_{\varepsilon>0}\text{Cl}\Big(\bigcup_{|\gamma'-\gamma|<\epsilon}D_{\gamma'}\Big)=D_\gamma.
\end{equation}
In particular, if $\mu$ is an ergodic measure supported on $\bar Z:=\text{Cl}\big(\bigcup Z_\gamma\big)$,
by Lemma \ref{lem:measures on Z}, the projection of $\mu$ on its first two factors
is supported on some $X_{e^\gamma}\times Y_\gamma$. By \eqref{eq:suppusc}, the only points in
$\bar Z$ projecting to $X_{e^\gamma}\times Y_\gamma$ are points in $Z_\gamma$. This completes the proof of Theorem \ref{thm:main+card}.

\section{Future Directions}\label{FutureDirections}
In this brief section, we take the opportunity to state a number of related questions that have arisen in the course of this investigation.
In our main theorem, we have identified necessary and sufficient conditions under which a function
$F\colon (\alpha,\infty)^m\to\R$ may be represented in the form $P_\text{top}(t_1\phi_1+\ldots+t_m\phi_m)$
for continuous functions $\phi_1,\ldots,\phi_m$ defined on a full shift.
A very natural question, suggested by an anonymous referee, is whether necessary and sufficient conditions
for such a representation can be identified for functions $F\colon \R^m\to\R$. As was mentioned in Remark \ref{rem:full shift}, an additional restriction on the value of $F$ at the origin is needed in this case and this value determines the full shift. Hence, the corresponding question in the spirit of the flexibility program is the following.

\begin{q}
Let $F$ be a Lipschitz convex function on $\R^m$ whose supporting hyperplane intercepts lie
in a bounded sub-interval of $\R_{\ge 0}$ such that $F(0,\ldots,0)=\log d$ for some integer $d\ge 2$.
Do there exist continuous functions $\phi_1,\ldots,\phi_m$ on the full shift on $d$-symbols  such that
$F(t_1,\ldots,t_m)=P(t_1\phi_1+\ldots+t_m\phi_m)$ for all $(t_1,\ldots,t_m)\in\R^m$?
\end{q}

A theorem of Israel \cite{Israel} shows that for any full shift and any subset $K$ of the set of the shift-invariant Borel probability
measures that is the weak$^*$-closure of the linear span of a non-empty collection of ergodic measures, there
exists a potential $\phi$ whose equilibrium states are precisely $K$. One may ask 
 about extensions of that theorem
in the spirit of the results in the current paper. A possible question in this direction would be the following which concerns
not only what are the possible pressure functions, but what are the possible pressure functions and associated
equilibrium states.

\begin{q}
Let $0<t_1<\ldots<t_n$ and let $P_1,\ldots,P_n$ be such that there is a convex function passing through
$(t_j,P_j)_{j=1}^n$. Let $K_1,\ldots,K_n$ be disjoint weak$^*$-closures of linear spans of collections of
ergodic invariant measures. Does there exist a continuous potential $\phi$ such that
$P(t_j\phi)=P_j$ for $j=1,\ldots,n$ and such that the equilibrium states of $t_j\phi$ are precisely $K_j$?
\end{q}

This paper has focused on the case where the potentials are required to be continuous, but with no stronger conditions.
Much of the development in thermodynamic formalism has concerned H\"older continuous functions.
It is therefore natural to ask for the analogous picture when the potentials are required to be H\"older continuous.

\begin{q}
Is it possible to give a nice characterization of those functions $F\colon\R^m\to\R$ that arise
as $P_\text{top}(t_1\phi_1+\ldots+t_m\phi_m)$ where $\phi_1,\ldots,\phi_m$ are required
to be H\"older continuous?
\end{q}

In her work on the flexibility program, Erchenko \cite{E2} has raised the question of which functions $F(t)$ arise as
$P_\text{top}(t\phi;T)$ where $T$ is a given Anosov diffeomorphism of $\mathbb T^2$
and $\phi$ is the geometric potential, that is
$\phi(x)=-\log |D_uT(x)|$ where $|D_uT(x)|$ is the derivative of $T$ in the unstable
direction. A possible approach to this question is to model the Anosov diffeomorphism
by a shift of finite type and address the two questions: (i) what are the possible pressure
functions corresponding to H\"older continuous functions on the shift of finite type; (ii)
of these pressure functions, which ones arise as the pressure function of a geometric
potential?

\end{document}